\documentclass[11pt,reqno]{amsart} 

\usepackage[utf8]{inputenc}   
\usepackage[OT1]{fontenc}      

\usepackage[margin=3cm]{geometry} 
\usepackage[skip=0.5\baselineskip, indent=0pt]{parskip}

\usepackage{mathtools}        
\usepackage{amssymb}          
\usepackage{esint}            

\usepackage{sansmathfonts}    
\usepackage{lmodern}          

\theoremstyle{plain} 
\newtheorem{theorem}{Theorem}

\newtheorem*{proposition*}{Proposition}
\newtheorem{lemma}[theorem]{Lemma}
\newtheorem{corollary}[theorem]{Corollary}

\theoremstyle{definition} 

\newtheorem{remark}[theorem]{Remark}

\usepackage{booktabs}       
\usepackage{placeins}
\usepackage{nicefrac}       
\usepackage{microtype}      
\usepackage[normalem]{ulem} 
\usepackage{enumerate}      
\usepackage{graphicx}       
\usepackage{lipsum}         
\usepackage{todonotes}      

\usepackage{url}            
\usepackage{hyperref}       
\hypersetup{hidelinks,
pdftitle={Quantitative equidistribution of eigenvalues of Random Normal Matrices in the Wasserstein distance},
pdfauthor={Pablo García Arias}
}
\usepackage{doi}            
\usepackage{cleveref}       

\usepackage[backend=biber, style=ieee, sorting=nyt, citestyle=numeric-comp]{biblatex}
\DeclarePairedDelimiter{\norm}{\lVert}{\rVert} 
\usepackage[sans]{dsfont}
\newcommand{\carac}{\mathds{1}}
\renewcommand{\Re}{\operatorname{Re}}

\newcommand{\Var}{\operatorname{Var}}
\newcommand{\Vol}{\operatorname{Vol}}

\makeatletter

\renewcommand{\section}{\@startsection{section}{1}{\z@}%
  {-3.5ex \@plus -1ex \@minus -.2ex}
  {2.3ex \@plus .2ex}
  {\normalfont\Large\bfseries\sffamily}} 

\renewcommand{\subsection}{\@startsection{subsection}{2}{\z@}%
  {-3.25ex\@plus -1ex \@minus -.2ex}
  {1.5ex \@plus .2ex}
  {\normalfont\large\bfseries\sffamily}} 

\renewcommand{\@secnumfont}{\bfseries\sffamily}

\makeatother

\title[\rmfamily Equidistribution of Random Normal Matrices]{\rmfamily Quantitative equidistribution of eigenvalues of Random Normal Matrices in the Wasserstein distance}
\author{\rmfamily Pablo García Arias}
\address{Departament de Matemàtiques i Informàtica, Universitat de Barcelona}
\email{pablo.garcia.arias@ub.edu}
\thanks{I would also like to thank Joaquim Ortega-Cerdà for his insights during the writing of this paper.}

\begin{document}

\begin{abstract}
	The object of study in this paper is the expected $2$-Wasserstein distance between the empirical measures of several point processes and their respective limit. For this, the main tool developed is a smoothing procedure in Euclidean spaces using the heat equation with Neumann boundary conditions. It is applied to the spectrum of Random Normal Matrices with \textit{reasonable} assumptions on the potential, as well as to Hyperuniform Point Processes such as the infinite Ginibre ensemble and the zero set of the planar Gaussian Analytic Function. In both of these cases, the technique obtains the optimal rate of convergence.
\end{abstract}

\maketitle 

\section{Introduction}

Point processes are used to construct (random) empirical measures that approximate a standard measure, such as the Lebesgue measure, but can be better computed numerically. Quantifying and bounding the error made by this substitution is a central topic in the field of study. 

A class especially relevant is the Determinantal Point Processes, as they are used in mathematical physics to model fermionic behaviour, they often appear in Random Matrix Theory, and they tend to produce well uniformly distributed points. The latter is due to an inherent repulsion between the different points, and quantifying it is the focus of this paper.

The main ensemble of point processes will be the spectrum of a Random Normal Matrix, for a given potential weight. The model instructs the point process with a determinantal structure, as well as making the eigenvalues accumulate in a compact set. The other class of point processes that is studied is the one with a homogeneous structure, that is, the first intensity is constant. Here the hypothesis of hyperuniformity will also play a relevant role to obtain sharp asymptotics. 

The Wasserstein distance is a metric over the space of probabilities on a metric space that metrizes the weak topology, so it can be used to quantify the equidistribution caused by the repulsion. Classically, this was done by directly bounding the discrepancy between the empirical measure and the reference measure, but in recent years the analogous question for the Wasserstein has been a central topic of study of the area.

To bound the Wasserstein distance, a scheme that can be applied is to smooth the empirical measures, so that they are not so ``rough'' and can be better compared with the limit measure. This is usually done through the use of Berry-Essen-like inequalities. This procedure was applied in \cite{borda2023riesz, arias2024equidistributionpointsharmonicensemble} to get optimal asymptotics in compact manifolds without boundary. Here an analogous smoothing inequality is presented for compact convex sets of the Euclidean space and its application to different point processes. Namely, upper bounds have been obtained for the expected value of some homogeneous point processes and random normal matrices models with \textit{reasonable} potentials. The asymptotics obtained are optimal for the random normal matrices and Hyperuniform point processes of Type III.

The paper is structured as follows: in Section \ref{sec:background} the relevant context is explained, in Section \ref{sec:results} the new results are stated. Section \ref{sec:proofs} contains the proofs as well as the required previous results in the literature.

\section{Background} \label{sec:background}
Let $\Omega \subseteq \mathbb{R}^d$ be a bounded open domain with Lipschitz boundary. It is well-known that the Laplacian $\Delta = \sum_{n=1}^d \frac{\partial^2}{\partial x_n^2}$ with Neumann boundary conditions has a basis of eigenfunctions $\Delta \phi_k = - \lambda_k \phi_k $ for the eigenvalues $0 = \lambda_0 \leq \lambda_1 \leq \lambda_2 \nearrow  \infty$ that decompose $L^2(\Omega)$. The eigenfunctions $\phi_k: \overline{\Omega} \to \mathbb{R} \in L^2(\Omega) \cap C(\overline{\Omega})$ are orthogonal and we will normalize them by $\norm{\phi_k}_{L^2(\Omega)} = 1$. For a measure $\mu$ on $\Omega$ we will consider its Fourier coefficients by $\widehat{\mu}(k) = \mu(\phi_k) = \int_\Omega \phi_k d \mu$. When the measure is supported in the whole space and the choice of restriction is clear, we will abbreviate $\mu|_\Omega (\phi_k) = \mu(\phi_k)$.   

We will also use the notation $dA(z) = dA(x+iy) = \frac{1}{\pi} dxdy$ for the normalized Lebesgue measure on the plane. 

\subsection{Wasserstein distance}

The Wasserstein distance in a metric space $(M, d)$ is defined for two finite measures $\mu, \nu$ as
\begin{align*}
    W_p(\mu, \nu) = \left( \inf \left\{ \int_{M\times M} d(x, y)^p d \pi(x, y) \right\} \right)^{\frac{1}{p}}, && 1 \leq p < \infty,
\end{align*} 
where the infimum is taken over all transport plans, that is, measures $\pi$ on $M\times M$ where the first and second marginals are $\mu$ and $\nu$. For this quantity to be finite is necessary that both measures have the same mass $\mu(M) = \nu(M)$, otherwise there would not exist any transport plan $\pi$. This is a distance in the set of finite measures with finite $p$--moment. 

A major property of this metric is that it induces the weak topology, that is, a sequence converges in the weak-$*$ sense $\mu_n \rightharpoonup \mu$ if and only if $W_p (\mu_n, \mu) \to 0$. Hence, if we have such a limit, one can use the Wasserstein distance to quantify the speed of convergence.

A particularly interesting case is when taking $p=2$, due to Brenier's theorem allowing to consider a subclass of transport plans. When the second measure $\nu$ is absolutely continuous, the infimum may be taken over all measurable functions $T:M \to M$ satisfying 
\begin{align*}
    W_2^2(\mu, \nu) = \inf \left\{ \int_M d(x, T(x))^2 d \nu(x) \;:\;  \nu(T^{-1}(B)) = \mu(B)\right\}.
\end{align*}    

In \cite{peyre_2018}, Peyre shows that the Wasserstein distance $W_2$ can be controlled by negative Sobolev norms. Given a measure $\mu$, define the seminorm 
\begin{align*}
    \norm{f}_{\dot{H}^1(d \mu)} = \left( \int |\nabla f|^2 d \mu \right)^{\frac{1}{2}},
\end{align*}
and for a signed measure $\gamma$, define by duality
\begin{align*}
    \norm{\gamma}_{\dot{H}^{-1}(d \mu)} = \sup \left\{ \left| \int f \; d \gamma \right| \;:\; \norm{f}_{\dot{H}^1(d\mu)} \leq 1 \right\}.
\end{align*}

 Then \cite[Theorem 1]{peyre_2018} says that for two (positive) measures 
\begin{equation} \label{eq:peyre}
     W_2(\mu, \nu) \leq 2 \norm{\mu - \nu}_{\dot{H}^{-1}(d\mu)}
\end{equation}

A relevant consequence of this result consists in bounding the Wasserstein distance by the $L_2$ norm of the difference, given that both measures are absolutely continuous. Indeed, if $\mu = \mu(x)dx$ and $\nu = \nu(x) dx$ have the same mass in the set $\Omega$ where we are computing the transport plans, and we can assume $a = \inf \mu(x) > 0$, then 
\begin{align*}
    W_2 (\nu, \mu) \leq 2 \norm{\nu - \mu}_{H^{-1}(d\mu)} \leq \frac{2}{\sqrt{a}} \norm{\nu - \mu}_{H^{-1}(dx)}.
\end{align*}    
We will assume that $\partial \Omega$ is regular enough\footnote{To consider the normal derivative, it would be enough to ask for a Lipschitz domain. Nonetheless we will apply it in piecewise-smooth sets.}, so that the norm is achieved by $\norm{\mu - \nu}_{H^{-1}(dx)} =  \norm{\nabla u}_{L^2}$ for $u$ the solution of the Neumann problem
\begin{align*}
	\left. \begin{matrix}
    - \Delta u = \mu - \nu &\text{in  } \Omega \\
    \nabla u \cdot n = 0 &\text{on  } \partial\Omega
	\end{matrix} \right\}
\end{align*}
The negative Sobolev seminorm can be bounded as
\begin{align*}
    \norm{\nabla u}^2_{L^2} = \int |\nabla u |^2 = \int (\mu-\nu) u \leq \norm{\mu - \nu}_{L^2} \norm{u}_{L^2}.
\end{align*}
Because $u$ is defined up to additive constants, Poincaré's inequality says $\norm{u}_{L^2} \leq C(\Omega) \norm{\nabla u}_{L^2}$. This gives
\begin{align*}
    \norm{\mu - \nu}_{H^{-1}(dx)} \leq C(\Omega) \norm{\mu - \nu}_{L^2}.
\end{align*} 
Notice that writing the first Laplacian eigenvalue as
\begin{align*}
    \lambda_1 = \inf_{u \neq 0} \frac{\int |\nabla u |^2}{\int u^2} & \Rightarrow \int |u|^2 \lesssim \frac{1}{\lambda_1} \int |\nabla u|^2
\end{align*} 
implies that we can take $C(\Omega) \leq \frac{1}{\sqrt{ \lambda_1(\Omega)}}$. Moreover, if we assume that $ \Omega $ is convex, then the Payne-Weinberger states that $ \lambda_1 (\Omega) \geq \frac{\pi^2}{\operatorname{diam}(\Omega)^2}$, so we will take $ C(\Omega) \lesssim \operatorname{diam}(\Omega) $.

All in all, given two absolute continuous measures $\mu, \nu$ on $\Omega \subseteq \mathbb{R}^2$ with one bounded below $\inf_\Omega \mu(x) = a > 0$ then
\begin{align} \label{eq:dynamicWass}
	W_2(\mu, \nu) \lesssim \frac{\operatorname{diam}(\Omega)}{\sqrt{a}} \norm{\mu-\nu}_{L^2(\Omega)}.
\end{align} 

Two standard references for a more detailed explanation of Optimal Transport and Wasserstein distance are \cite{santambrogio_2015, villani_2009}.

\subsection{Point Processes} \label{sec:pointprocesses}

A (simple) point process in a topological space $M$ is a random variable $\Xi$ that outputs subsets of $M$ without accumulation points. They are often characterized by their joint intensities: the family of non-negative functions $\rho_k: M^k \to [0, \infty)$ satisfying
\begin{align*}
    \mathbb{E} \sum_{x_1, \dots, x_k \in \Xi \text{ distinct}} f(x_1, \dots, x_k) = \int f(x_1, \dots, x_k) \rho_k(x_1, \dots, x_k) d \nu(x_1) d\nu(x_2) \dots d\nu(x_k)
\end{align*} 
for any measurable  $f$. The background measure is denoted by $\nu$, and it is often some normalization of the Lebesgue measure. 

The sum $\sum_{x_1, \dots, x_k \in \Xi} f(x_1, \dots, x_k)$ over all tuples of indices $(i_1, i_2, \dots, i_k)$ without repetition, to which we are computing the expectation, is called a statistic of the point process. An especially interesting case is linear statistics, that is, the random variable $\sum_{x \in \Xi} f(x)$. 

A point process $\Xi$ is said to be homogeneous when the first intensity $\rho_1$  is constant. For it to be of interest, the constant must be strictly positive so that the point process is not empty. This will be tacitly assumed. Intuitively, the constant first intensity means that the process ``looks the same'' everywhere.

One can rescale such point process by a factor of $L^{-\frac{1}{d}}$ to get a family of homogeneous point processes $\Xi_L$ with first intensity $\rho_1^L(x) = cL$. This concentration of the points can also be seen as an expansion of the target domain. Indeed, let $A \subseteq  \mathbb{R}^d$ be the compact set one is working with, then
\begin{align*}
    x \in \Xi_L \cap A \iff L^{\frac{1}{d}}x = y \in \Xi \cap (L^{\frac{1}{d}} A).
\end{align*} 

For such point processes we will fix a set $\Omega$ and consider the sequence of (normalized) empirical measures 
\begin{align*}
    \mu_L = \frac{1}{N}\sum_{x \in \Xi_L \cap \Omega} \delta_x, && N= \# \Xi_L \cap \Omega.
\end{align*}
An example of this rescaling construction is the zero set of the planar GAF. The planar Gaussian Analytic Function is the random holomorphic function
\begin{align*}
    f_L(z) = \sum_{n=0}^{+\infty} a_n \sqrt{ \frac{L^n}{n!} } z^n, && a_n \sim N_\mathbb{C}(0, 1) \text{  i.i.d.}
\end{align*}  
which has covariance kernel $K(z, w) = e^{L z \overline{w}}$. One may consider the point process given by the random zero set $\mathcal{Z} ( f_L) = \left\{ z_1, z_2, \dots \right\}$. It is known to be a homogeneous point process with first intensity $\rho_1(z) = \frac{L}{\pi} dz$. 

\begin{figure}[ht]
    \centering
    \includegraphics[width=0.3\linewidth]{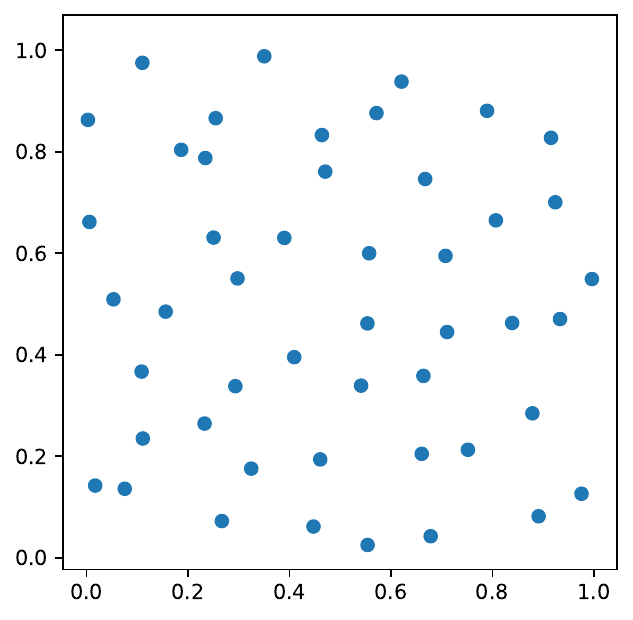}~
    \includegraphics[width=0.3\linewidth]{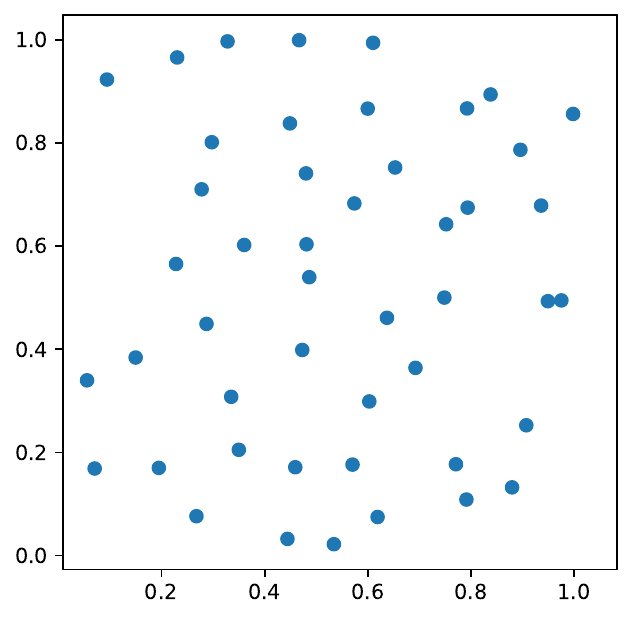}
    \caption{To the left, a sample from the zero set of the planar GAF, restricted to the unit square, and intensity $L=150$. To the right, a sample from the infinite Ginibre ensemble restricted to the unit square, and intensity $L=150$.}
\end{figure}

A particularly studied class of point processes is that of Determinantal Point Processes (DPPs). They are the ones whose joint intensities are given by a determinant $\rho_k (x_1, \dots, x_k) = \det(K(x_i, x_j))$ for some kernel $K: M\times M \to \mathbb{C}$ that we will further assume to be Hermitian.

If the kernel is locally square integrable on $M^2$ one can consider the operator 
\begin{align*}
    \mathcal{K}: L^2(M) & \longrightarrow L^2(M) \\
    f & \longmapsto \mathcal{K} f (x) = \int K(x, y) f(y) d\nu(y).
\end{align*} 
When $\mathcal{K}$ is a projection operator over some $L^2$ vector subspace, we will say that $\Xi$ is a projection DPP. If this subspace has the form $\overline{\operatorname{span}}  \left\{ f_n \right\}$ with $\left\{ f_n \right\}$ orthonormal then the kernel $K$ can be written $K(x, y) = \sum f_n(x) \overline{f_n(y)}$ almost everywhere. 

For linear statistics of a general DPP, the expected value can be readily expressed as 
\begin{align*}
    \mathbb{E} \sum_{x \in \Xi} f(x) = \int f(x) K(x, x) d\nu(x).
\end{align*}
Projection DPPs satisfy the following reproducing property for any point $x$:
\begin{align}\label{eq:reproducingprop}
    K(x , x) = \int |K(x, y)|^2 d \nu (y).
\end{align}
Using this, the variance has the useful formula
\begin{align*}
    \Var \sum_{x \in \Xi} f(x) = \frac{1}{2}\iint |f(x) - f(y)|^2 |K(x, y)|^2 d\nu(x)d\nu(y).
\end{align*}

An example is the infinite Ginibre ensemble: it is the DPP on the complex plane $\mathbb{C}$ with the projection kernel
\begin{align*}
    K_L(z, w) = \frac{L}{\pi} e^{-L \frac{|z|^2+|w|^2-2z\overline{w}}{2}}
\end{align*}
with respect to the Lebesgue measure as background. Notice that it is a homogeneous point process to which one is applying the previous concentration procedure to the case $L=1$.

The point processes mentioned before have especially good equidistribution properties, due to their hyperuniformity. More formally, a homogeneous point process $ \Xi $ is said to be hyperuniform when it is invariant under translations and 
\begin{equation*}
	\lim_{r \to \infty} \frac{\Var \left( \Xi(B_r) \right) }{\operatorname{Vol}(B_r)} = 0. 
\end{equation*}

This class of point processes is often characterized and studied on the Fourier side. For such point processes, it can be shown there is a non-negative measure $ S $ called the spectral measure or Bartlett  spectrum such that
\begin{equation*}
	\Var \left(\sum_{x \in \Xi} f(x)\right) = \int_{\mathbb{R}^d} \left| \widehat{f}(\xi) \right|^2 dS(\xi).  
\end{equation*}
For the Fourier transform, we are taking the convention $ \widehat{f}(\xi) = \int f(x) e^{- 2 \pi i x \cdot \xi} dx $ for $ L^1 $ functions. When this measure is absolutely continuous, its density $ s $ is called the structure factor. The hyperuniformity condition is in that case equivalent to $ s(0) = 0 $. There are 3 distinguished types of hyperuniformity depending on the behaviour of $ s $ near the origin:
\begin{itemize}
	\item Type I: $ s(w) = c |w|^\alpha + o(|w|^\alpha) $ as $ w \to 0 $, for some $ \alpha > 1 $. This implies that the variance behaves like $ \Var (\Xi(B(0, R))) = O(R^{d-1})$.
	\item Type II: $ s(w) = c |w|+ o(|w|) $ as $ w \to 0 $. The variance has now the asymptotic $ \Var (\Xi(B(0, R))) = O(R^{d-1}\log R)$.
	\item Type III: $ s(w) = c |w|^\alpha + o(|w|^\alpha) $ as $ w \to 0 $, for some $ 0 < \alpha < 1 $. They then satisfy $ \Var (\Xi(B(0, R))) = O(R^{d-\alpha})$.
\end{itemize}
It is important to highlight that this is not an exhaustive categorization of all hyperuniform point processes. When considering the dilation of the point process, the variance of linear statistics is expressed as 
\begin{equation*}
	\Var \left( \sum_{x \in \Xi_L} f(x) \right) = L \int_{\mathbb{R}^d} \left| \widehat{f}(\xi) \right|^2 s\left( \frac{|\xi|}{L^{ \frac{1}{d} }}  \right)  d\xi.  
\end{equation*}

The function $ s $ is also uniformly continuous and bounded. As a consequence, we can bound $ s(w) \approx \min ( |w|^\alpha, 1 ) $ in each of the three types. Both the zero set of the planar GAF ($ \alpha = 4 $) and the infinite Ginibre ensemble ($ \alpha=2 $) are of type I.

For a more detailed introduction to the theory behind point processes, see the book \cite{Hough_Krishnapur_Peres_2012}. A survey more focused on hyperuniformity is \cite{coste_2021}.

The other main point process we will study comes from random matrix theory. Given an upper semicontinuous function $Q : \mathbb{C} \to \mathbb{R}$ called the potential, one can associate to it a random normal matrix model consisting of the $n \times n$ complex normal matrices, that is, the ones satisfying $M^*M=MM^*$, and the probability measure
\begin{align*}
    \frac{1}{\mathcal{Z}_n} e^{-n \operatorname{tr} Q(M)} dM,
\end{align*}  
where $dM$ denotes the standard Riemannian volume form on the manifold of the normal matrices, and $\mathcal{Z}_n$ is the appropriate normalizing constant. $\operatorname{tr} Q(M)$ can be interpreted as the sum of $Q$ evaluated on the eigenvalues of the normal matrix $M$.

To ensure integrability and the existence of $\mathcal{Z}_n$ one has to impose some growth condition on $Q$, being the standard one that 
\begin{align}\label{eq:growth}
    \liminf_{|z| \to \infty} \frac{Q(z)}{\log |z|^2} > 1.
\end{align}  

The random measure one is usually interested in is not this one on the normal matrices but rather the one induced by considering their eigenvalues, as they are (almost surely) $n$ distinct points in $\mathbb{C}$. This turns out to be a DPP with distribution
\begin{align*}
    d \mathbb{P}_n  (z_1, \dots, z_n) = \frac{1}{Z_n} \prod_{1 \leq i < j \leq n} \hspace{-0.3em}  |z_i - z_j|^2\hspace{-0.3em} \prod_{1 \leq j \leq n} \hspace{-0.3em}  e^{-n Q(z_j)} dA(z_j) = \det \left( K_n(z_i, z_j) \right) dA(z_1)\hspace{-2pt} \dots \hspace{-2pt}  dA(z_n).
\end{align*}
For the kernel we will choose 
\begin{align*}
    K(z, w) = K_n(z, w) = e^{-\frac{1}{2} n (Q(z)+Q(w))} \sum_{j=0}^{n-1} p_j(z) \overline{p_j(w)},
\end{align*}
where $p_j$ is the planar orthogonal polynomial of degree $j$ and positive leading coefficient satisfying
\begin{align*}
    \int p_j(z) \overline{p_k(z)} e^{-nQ(z)} dA(z) = \delta_{j, k}.
\end{align*}

Notice that this is a reproducing kernel, so it will induce a projection DPP.

The study of random matrices, their eigenvalues, or similar constructions has been of interest to mathematicians and physicists for a while, see for instance the seminal work of Ginibre \cite{Ginibre1965} in 1965. The general mathematical structure used here probably dates back to the decade of the 2000s, appearing in works by Elbau and Felder \cite{elbau_felder_2005} and by Hedenmalm and Makarov \cite{hedenmalm_makarov_2012}. 

Under fairly general conditions on the potential, the eigenvalues tend to concentrate in a compact set as the number $n$ goes to infinity. This compact set $S$ is called the droplet and can be constructed using potential theory. As there is a finite number of points, the normalized empirical measure is defined as 
\begin{align*}
    \mu = \frac{1}{n} \sum_{i=1}^n \delta_{z_i}
\end{align*}
 which in the limit $n \to \infty$ converges weakly to the equilibrium measure $d\sigma = \frac{1}{4} \Delta Q \carac_S dA$.

Let the function $\widehat{Q}$, called the equilibrium potential, be given by
 \begin{align*}
   \widehat{Q} (z) =  \sup \left\{ f(z) \;:\; f:\mathbb{C} \to \mathbb{R}, \text{  subharmonic  }, f(w) \leq Q(w), f(w) \leq \log_+ |w|^2 + O(1) \right\}.
 \end{align*}
It can be checked that $\widehat{Q}(z) = \log_+|z|^2 + O(1)$ and is harmonic outside of a compact set $S$, which turns out to be the droplet. The equilibrium measure then can be expressed as 
\begin{align*}
    d \sigma(z) = \frac{1}{4} \Delta \widehat{Q}(z) dA(z) = \frac{1}{4} \Delta Q(z) \carac_S(z) dA(z).
\end{align*} 

One can also consider the coincidence set $S^* = \left\{ Q = \widehat{Q} \right\}$. This is called the pre-droplet. As expected, $S \subseteq S^*$, and if one assumes $Q$ is $C^2$ in the difference set $S^* \backslash S$ then one has $\Delta Q (z) = 0$ almost everywhere in $S^* \backslash S$. 

The accumulation of the spectrum can be seen in expectation too. Indeed, notice $\mathbb{E} \mu$ is an absolutely continuous measure w.r.t. Lebesgue whose density has the behaviour
\begin{align*}
    \frac{1}{n} \rho_1(z) = \frac{1}{n} K(z, z) \xrightarrow{n \to \infty} \left\{ \begin{matrix}
		    \frac{1}{4}\Delta Q(z) & z\in \mathring{S} \\[0.6em]
		    \frac{1}{8} \Delta Q(z) & z \in \partial S \\[0.6em]
        0 & z \notin S \\
    \end{matrix} \right.
\end{align*}

A classic example of the Random Matrix Model is the Ginibre ensemble. It corresponds to taking $Q(z) = |z|^2$. The droplet is then the unit ball $S = \overline{\mathbb{D}}$. Another way of seeing this probability distribution is to consider the eigenvalues of the random $n\times n$  matrix whose entries are i.i.d. complex Gaussians and then rescale them by a $\sqrt{n}$ factor. A generalization of this process is the Elliptic Ginibre ensemble, which is the one with the potential $Q(z) = \frac{1}{1-\tau^2} \left( |z|^2 - \frac{\tau}{2} (z^2 + \overline{z}^2) \right)$ for $0 \leq \tau < 1$. 

A more comprehensive explanation of the RNM model can be found in \cite{WardIdentities, MarzoUniversality}, among others.

\begin{remark}\label{remark:regularity}
    Our regularity assumptions on the potential are the following:
    \begin{itemize}
	\item Growth condition \eqref{eq:growth}.
	\item $ Q\in C^2(\mathbb{C}) $.
        \item $Q$ is real-analytic in some neighbourhood of the droplet $S$.
        \item $\Delta Q \neq 0$ in $S$.
	\item $\partial S$ is a $C^\omega$ smooth Jordan curve.  
    \end{itemize}
They are chosen mainly so that we can apply some results from the papers \cite{BerezinTransform, ameur_2011} that we will require.
\end{remark}
\begin{figure}[!h]
    \centering
    \includegraphics[height=0.3\linewidth]{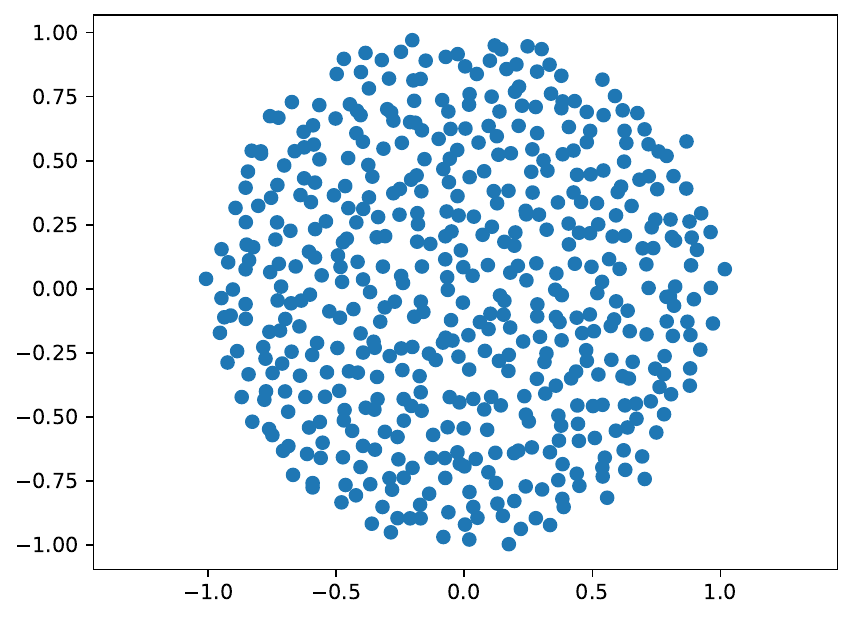}
    \includegraphics[height=0.3\linewidth]{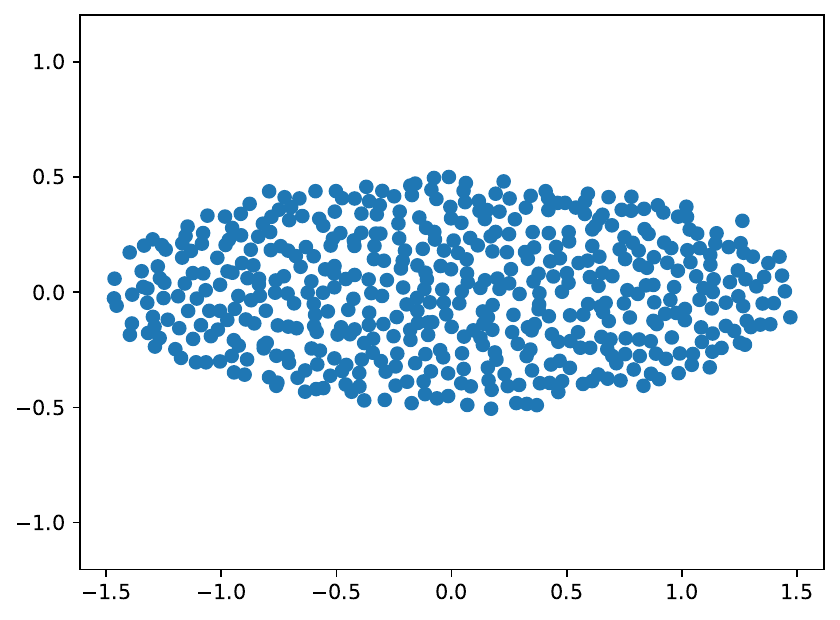}
    \caption{Sample of 500 points of the Ginibre ensemble (left) and the Elliptic Ginibre ensemble (right).}
    \label{fig:rnm}
\end{figure}

The relationship between the Wasserstein distance and point configuration has been influenced in great part by the following classic lower bound. As is the case with the empirical measures, when one has a purely atomic probability measure $\mu$, supported at $n \in \mathbb{N}$ points, and an absolutely continuous probability measure $\sigma$, with density bounded below strictly from zero $\sigma(z) \geq c > 0$ in its support, then
\begin{equation*}
    W_p(\mu, \sigma) \geq W_1(\mu, \sigma) \geq C \frac{1}{N^{ \frac{1}{d} }},
\end{equation*}
with $d$ the dimension and the constant $C$ depending on $d$ and $\sigma$ only. Hence, when studying the asymptotics of $\mathbb{E} W_2$, one is often focused on the upper bounds. 

It is worth commenting on the case of the Poisson Point Process, that is, the case of independent identically distributed points. When $x_1, x_2, \dots, x_n$ are independent and uniformly distributed on the cube $[0,1]^d$, they behave like
\begin{align*}
    \mathbb{E} W_2 \left( \frac{1}{N} \sum \delta_{x_i}, \operatorname{Vol}|_{[0,1]^d} \right) \sim \left\{ \begin{matrix}
        \displaystyle\frac{\sqrt{\log n}}{\sqrt{n}},& \quad d=2, \\[1em]
        \displaystyle\frac{1}{N^{ \frac{1}{d} }}, & \quad d \geq 3.
    \end{matrix}\right.
\end{align*}

\section{Main results} \label{sec:results}
    
    The general outline is to apply a heat smoothing procedure over a convex set in order to better compare the empirical measure and the limit using the expected $2$-Wasserstein distance. This is developed in Section \ref{sec:smoothing}. It has application to several classes of point processes, notably to the Random Normal Matrices model and Homogeneous point processes with adequate bounds on the Variance.

    With respect to the related literature results, it is worth signalling that the previous results focused mostly on the Ginibre ensemble: Prod'Homme proved in his thesis \cite{prodhomme} the rate of convergence $O(\frac{1}{\sqrt{n}})$ for the expected $2$-Wasserstein distance; and Jalowy \cite{jalowyCircularLaw} studied several asymptotics for the $1$-Wasserstein distance as well as concentration of the probability in $W_p$. The Coulomb gases model studied in \cite{CHAFAI20181447} includes the RNM model, proving concentration results that imply an asymptotic of the order $ \mathbb{E} W_1 \lesssim \frac{\sqrt{\log n}}{\sqrt{n}} $.

    Depending on whether one studies the whole spectrum on the complex plane, or just locally on the droplet, the heat smoothing can be applied to obtain the two theorems below. 
    \begin{theorem}[Global version] \label{thm:RNM}
        Let $Q$ be a potential satisfying the regularity assumptions of Remark \ref{remark:regularity} and whose droplet $S$ is convex. If we denote by $z_1, z_2, \dots, z_n$ the points given by the RNM model, then
        \begin{align*}
		\mathbb{E} W_2\left( \frac{1}{n} \sum \delta_{z_i},  \frac{1}{4}\Delta Q\carac_{S} dA \right) \lesssim\frac{1}{\sqrt{n}}.
        \end{align*}
    \end{theorem}
    \begin{theorem}[Local version]\label{thm:RNMbulk}
        Let $Q$ be a potential satisfying the regularity assumptions of Remark \ref{remark:regularity}, and let $\Omega$ be a convex open set with positive measure and $\overline{\Omega} \subset \mathring{S}$, then
        \begin{align*}
            \mathbb{E} W_2\left( \frac{1}{\# \{z_i\in \Omega\} } \sum_{z_i \in \Omega} \delta_{z_i},  \frac{1}{4 \sigma(\Omega) }\Delta Q\carac_{\Omega} dA \right) \lesssim\frac{1}{\sqrt{n}}.
        \end{align*}
    \end{theorem}
    The connection between hyperuniform point processes and optimal transport has been deeply studied in recent years. For stationary point processes (the ones invariant by translations), the theory of Optimal Transport has been modified to work with infinite measures, see \cite{ERBAR2025110974} and references therein. More in the spirit of the results presented here, the paper \cite{hyperuniformityoptimaltransportpoint} also studies the limit of dilations, and proves an asymptotic similar to the following ones when $ \Omega = [0,1]^d $.
    \begin{theorem} \label{thm:generalthm}
        Let $\Xi$ be a homogeneous point process in $\mathbb{R}^d$ and $\Xi_L$ be a dilation of it by a factor of $L^{-\frac{1}{d}}$ as described in Section \ref{sec:pointprocesses}. Fix an open bounded convex set $\Omega$. We will further assume that for any function $f\in L^2(\Omega)$ the linear statistics satisfy
        \begin{align}\label{eq:varinequality}
                \Var \left( \sum_{x \in \Xi_L} f(x) \right) \lesssim L \norm{f}^2_{L^2(\Omega)}.
        \end{align}

        Let $N = \# \Xi_L \cap \Omega$. Then we have the following asymptotic:
        \begin{align*}
            \mathbb{E} W_2 \left( \frac{1}{N} \sum_{\Xi_L \cap \Omega} \delta_x, \frac{1}{\operatorname{Vol}(\Omega)} \operatorname{Vol}|_{\Omega} \right) \lesssim \left\{ \begin{matrix}
                \frac{\sqrt{\log L}}{\sqrt{L}} & d=2 \\[1em]
                \frac{1}{L^{\frac{1}{d}}} & d > 2
            \end{matrix}\right.
        \end{align*}
    \end{theorem}
    \begin{corollary}\label{thm:homogeneous}
        The assertions of Theorem \ref{thm:generalthm} hold for any homogeneous DPP satisfying the reproducing property \eqref{eq:reproducingprop}.
    \end{corollary}
    The above corollary immediately applies to ensembles such as the Infinite Ginibre one. Theorem \ref{thm:generalthm} can also be applied to the zero set of the planar GAF, given that \eqref{eq:varinequality} holds thanks to \cite{nazarov2010fluctuationsrandomcomplexzeroes}. While this result is interesting in high dimensions, for $ d=2 $ it does not give an optimal result. In dimension $2$, using hyperuniformity, one can be more precise and obtain the following asymptotic.
    \begin{theorem}\label{thm:hyperuniform}
    Let $ \Xi $ be a hyperuniform point process on $ \mathbb{R}^2 $ of type $ I $, $ II $ or $ III $. Fix any open bounded convex set $ \Omega $, and let $ N = \# \Xi_L \cap \Omega $ be the number of points in the set. Then
    \begin{align*}
	    \mathbb{E} W_2 \left( \frac{1}{N} \sum_{\Xi_L \cap \Omega} \delta_x, \frac{1}{\operatorname{Vol}(\Omega)} \operatorname{Vol}|_{\Omega} \right) \lesssim \frac{1}{\sqrt{L}}. 
    \end{align*}
    
\end{theorem}

\section{Proofs} \label{sec:proofs}

\subsection{Heat Smoothing inequality}\label{sec:smoothing}

The scheme is to apply a smoothing procedure to $\mu$ and $\nu$ using the Neumann heat equation to get absolutely continuous measures $\mu_t$ and $\nu_t$ \textit{which are close to the original ones} but can be better compared. 

This idea of smoothing an empirical measure using some kernel has actually been a staple technique for some time now. In dimension 1, the problem has been in the interval $[0,1]$ and the torus $\mathbb{T}^1$ for $W_p$ in the range $p \geq 2$. In higher dimensions, there have been similar studies for the unit cube, often identified as the $d$-dimensional torus, in different ranges of $p$. To name a few, it has been applied by Brown and Steinerberger \cite{brown_steinerberger_2020}, and by Bobkov and Ledoux \cite{bobkov_ledoux_2021}. This was generalized to compact Riemannian manifolds (without boundary) for the quadratic Wasserstein distance by Borda \cite{borda_2023}, with ramifications seen in \cite{borda2023riesz, arias2024equidistributionpointsharmonicensemble}. It is worth commenting that this approach has recently been generalized to $p \neq 2$ in \cite{borda2025smoothinginequalitiestransportmetrics}. 

Given a finite measure $u_0$ on $\Omega$ one may consider the heat equation with initial condition $u_0$ 
\begin{align*}
	\left.\begin{matrix}
    	\partial_t u = \Delta u & \text{in $\Omega$}\\
    	\frac{\partial u}{\partial n} = 0 & \text{on $\partial \Omega$}\\
	u_{t=0}=u_0 &
	\end{matrix}\right\}
\end{align*}  

The solution for $t>0$ can be expressed using the heat kernel $\displaystyle P(t, x, y) = \sum_{k\geq 0} e^{-\lambda_k t} \phi_k(x) \phi_k(y)$ as
\begin{align*}
    u_t(x) = \int_\Omega P(t, x, y) d u_0(y) = P_t[u_0](x).
\end{align*}

\begin{theorem} \label{thm:smoothing}
    Given $\mu, \nu$ measures on an open bounded convex domain $\Omega \subseteq \mathbb{R}^d$, with $\nu \geq c \Vol|_{\Omega}$ for some $c\geq 0$ and $\mu(\Omega) = \nu(\Omega)$ , then for any $t > 0$
    \begin{align*}
        W_2(\mu, \nu) \leq C_1 (dt)^{\frac{1}{2}} + \frac{2}{C_2(\nu, t)} \left( \sum_{k\geq 1} \frac{e^{-\lambda_k t}}{\lambda_k} | \widehat{\mu}(k) - \widehat{\nu}(k) |^2 \right)^{\frac{1}{2}}, 
    \end{align*}  
    where $C_1 = \mu(\Omega)^{\frac{1}{2}} + (\nu(\Omega) - c \Vol(\Omega))^{\frac{1}{2}}$, and $C_2(\nu, t) = \left( \inf_{x\in \Omega} \int P(\frac{t}{2}, x, y) d \nu(y) \right)^{\frac{1}{2}}$. 
\end{theorem}

\begin{proof}
    For $\mu, \nu$ two finite measures on $\Omega$ we can define the intermediate measures
    \begin{align*}
        \mu_t =  P_t[\mu], && \nu_t =  P_t[\nu].
    \end{align*}  
    As we are choosing Neumann conditions, the mass remains constant and the same in both measures. Therefore we will bound
    \begin{align*}
        W_2(\mu, \nu) \leq W_2(\mu, \mu_t) + W_2(\mu_t, \nu_t) + W_2(\nu_t, \nu).
    \end{align*}
    We will use two Lemmas to prove the result. For the difference between the original and the smoothed versions $W_2(\mu, \mu_t)$ and $W_2(\nu_t, \nu)$, we will exploit the properties of the heat kernel. For the distance $W_2(\mu_t, \nu_t)$ between the smoothed versions, we will bound it as a negative Sobolev norm, which then can be studied using Fourier expansion. Hence, the result follows from Lemma \ref{lemma:smoothclose} and Lemma \ref{lemma:smooth_fourier}.

    Note that we get a bound in terms of ``$2t$'', which then can be rescaled to get the exact inequality of the statement. 
\end{proof}

\begin{lemma}
    For any $0<t$ and $y\in \Omega$ 
    \begin{align*}
        \int_\Omega |x-y|^2 P(t, x, y) dx \leq 2 d t.
    \end{align*}  
\end{lemma}

\begin{proof}
    Fix the point $y\in \Omega$ and define for $t > 0$ the function
    \begin{align*}
        F(t) = \int_\Omega |x-y|^2 P(t, x, y) dx.
    \end{align*}
    Differentiating it with respect to $t$ we get
    \begin{align*}
        F'(t) & = \int_\Omega |x-y|^2 \frac{\partial}{\partial t}P(t, x, y) dx = \int_\Omega |x-y|^2 \Delta_x P(t, x, y) dx \\
         & = \int_\Omega \Delta_x |x-y|^2 \;\; P(t, x, y) dx + \int_{\partial \Omega} |x-y|^2 \partial_n P(t, x, y) dx - \int_{\partial \Omega} \partial_n |x-y|^2 P(t, x, y) dx.
    \end{align*} 
    Notice that because of the Neumann conditions $\partial_n P(t, x, y) = 0$ on $\partial S$. On the other hand, $\partial_n |x-y|^2 = 2 (x - y) \cdot n(x)$ and $\Delta |x-y|^2 = 2d$. The convexity assumption on $\Omega$ implies $(x-y) \cdot n(x) \geq 0$ for any $x\in \partial \Omega$ and $y\in \Omega$, so we have
    \begin{align*}
        F'(t) \leq 2d \int_{\Omega} P(t, x, y) = 2d.
    \end{align*}
    As $P(0, x, y) = \delta_y(x)$ (in the sense of distributions), we can take $F(0) = 0$. Therefore, integrating $F(t) \leq 2d t$.   
\end{proof}

The above pointwise bound can be used to get bounds for the distance between a general measure and its regularization. 
\begin{lemma} \label{lemma:smoothclose}
    In the conditions of Theorem \ref{thm:smoothing}, the error introduced by the intermediate measures is bounded by $W_2(\mu, \mu_t) \leq (2dt \mu(\Omega))^{\frac{1}{2}}$ and $W_2(\nu, \nu_t) \leq (\nu(\Omega) - c \Vol(\Omega))^{\frac{1}{2}} (2dt)^{\frac{1}{2}}$. 
\end{lemma}
\begin{proof}
    Define the measure $\pi$ on $\Omega\times \Omega$ by
    \begin{align*}
        d \pi(x, y) = P(t, x, y) d \mu(y) dx.
    \end{align*}  
    The first marginal clearly gives us $d \mu_t (x)$, while integrating with respect to $x$ requires using that $\int P(t, x, y) dx = 1$ and hence the second marginal is $d \mu (y)$. Thus $\pi \in \Pi(\mu, \mu_t)$ is an appropriate transport plan which gives the bound
    \begin{align*}
        W_2(\mu , \mu_t)  \leq \left( \iint_{\Omega\times \Omega} |x-y|^2 d \pi(x, y) \right)^{\frac{1}{2}} & = \hspace{-0.15em} \left( \iint |x-y|^2 P(t, x, y) dx d \mu(y) \right)^\frac{1}{2} 
        \\
        &\leq\left( \int_\Omega 2d t \;d \mu(y) \right)^{\frac{1}{2}}\hspace{-0.3em} \leq (2dt \mu(\Omega))^{\frac{1}{2}} .
    \end{align*}  

    For the transport corresponding to $\nu$, as we further assume the hypothesis $\nu \geq c \Vol|_{\Omega} $, we can leave a $c$ amount of mass without moving, so we get $W_2(\nu, \nu_t) \leq (\nu(\Omega)-c \Vol(\Omega))^{\frac{1}{2}} (2dt)^{\frac{1}{2}}$. More precisely, we are considering the transport plan
    \begin{align*}
        d \pi(x, y) = c d \lambda(x, y) + P(t, x, y) \left( d \nu(y) - c dy \right) dx,
    \end{align*}
    where $\lambda$ is the singular probability measure characterized by $\lambda(A) = \Vol\left( \left\{ x\in \Omega \;:\; (x, x) \in A \right\} \right)$.

\end{proof}

\begin{lemma}\label{lemma:smooth_fourier}
    In the conditions of Theorem \ref{thm:smoothing}, the smoothed measures satisfy 
    \begin{align*}
        W_2(\mu_t, \nu_t) \leq \frac{2}{C_2(\nu, 2t)} \left( \sum_{k\geq 1} \frac{e^{-\lambda_k 2t}}{\lambda_k} | \widehat{\mu}(k) - \widehat{\nu}(k) |^2 \right)^{\frac{1}{2}}.
    \end{align*}
\end{lemma}

\begin{proof}
	The main tool is \eqref{eq:peyre}, as it relates the Wasserstein distance $W_2$ with negative Sobolev norms: 
    \begin{align*}
        W_2(\mu_t, \nu_t) \leq 2 \norm{\mu_t - \nu_t}_{\dot{H}^{-1}(\nu_t)} = 2 \sup \left\{ \left|\langle f, \mu_t-\nu_t \rangle \right| \;:\; \norm{f}_{\dot{H}^1 (\nu_t)} \leq 1 \right\}.
    \end{align*}
    
    Notice that the supremum is the norm of a linear operator, so it is enough to consider a dense subspace of functions. Given $f\in C^\infty(\Omega)$ with $\partial_n f = 0$ on the boundary, the norm can be expressed as
    \begin{align*}
        \norm{f}^2_{\dot{H}^1(\nu_t) } & = \int_{\Omega} |\nabla f(x)|^2 \int_\Omega P(t, x, y) d \nu(y) dm(x) \geq \hspace{-1.6pt} \left( \inf_{x\in \Omega}  \int P(t, x, y) d \nu(y) \right) \hspace{-1.6pt} \int_{\Omega} |\nabla f(x) |^2 dm(x) \\
        & = \left( c_2(\nu, 2t) \right)^2 \left[ \int_\Omega -\Delta f \cdot f + \int_{\partial \Omega} f \frac{\partial f}{\partial n} \right] = \left( C_2(\nu, 2t) \right)^2 \sum_{k>0} \lambda_k |\widehat{f}(k)|^2.
    \end{align*}
    
    So, if $\norm{f}^2_{\dot{H}^1(\nu_t) } \leq 1$ then $\displaystyle  \left( \sum_{k>0} \lambda_k |\widehat{f}(k)|^2 \right)^{\frac{1}{2}} \leq \frac{1}{C_2(\nu, 2t)}$. This implies
    
    \begin{align*}
        \left|\int_\Omega f d(\mu_t-\nu_t)\right| & = \left| \iint_{\Omega\times \Omega} f(x) P(t, x, y) d(\mu(y)-\nu(y)) dm(x)  \right| \\
        & = \left| \sum_{k\geq 1} e^{- \lambda_k t} \int_\Omega f(x) \phi_k(x) dm(x) \int \phi_k(y) (d \mu(y) - d \nu(y)) \right| \\
        &= \left| \sum_{k\geq 1} e^{- \lambda_k t} \widehat{f}(k) \left( \widehat{\mu}(k)- \widehat{\nu}(k) \right) \right| \\
        & \leq \left( \sum_{k\geq 1} \lambda_k |\widehat{f}(k)|^2 \right)^{\frac{1}{2}} \left( \sum_{k\geq 1} \frac{e^{-\lambda_k 2t}}{\lambda_k} | \widehat{\mu}(k) - \widehat{\nu}(k) |^2 \right)^{\frac{1}{2}} 
        \\ &\leq \frac{1}{C_2(\nu, 2t)} \left( \sum_{k\geq 1} \frac{e^{-\lambda_k 2t}}{\lambda_k} | \widehat{\mu}(k) - \widehat{\nu}(k) |^2 \right)^{\frac{1}{2}}.
    \end{align*}
\end{proof}

Notice that if there is \textit{some regularity}, i.e. $\nu \geq c \Vol|_\Omega $ for $c>0$, then $C_2( \nu , 2t) > c^{\frac{1}{2}}$ and can be bounded independently of $t$. Indeed, this is just saying that 
\begin{align*}
    C_2(\nu, 2t)^2 = \inf_{x\in \Omega} \int P(t, x, y) d \nu(y) = \inf_{x\in \Omega} \int P(t, x, y) \left( d \nu(y) - c dy \right) + c \geq c.
\end{align*} 

In order to apply this result easily in the setting of point processes, it may be useful to have the following corollary.
\begin{corollary} \label{cor:smoothing}
    Let $\mu_L$ be a sequence of measures on an convex domain $\Omega \subseteq  \mathbb{R}^d$ and $\nu_L$ another one for which we will assume $\nu_L \geq c dm|_{\Omega}$ for a constant $c>0$ independent of the index $L$ and $ \mu_L(\Omega) = \nu_L(\Omega)$. Assume the following estimates
    \begin{align*}
        \mathbb{E} \left( \left|\mu_L(\phi_k) - \nu_L(\phi_k)\right|^2 \right) \lesssim \frac{\lambda_k^b}{L^a} , && b+ \frac{d}{2} > 1.
    \end{align*}
    Then
    \begin{align*}
        \mathbb{E} W_2\left( \mu_L, \nu_L \right) \lesssim \frac{1}{L^\gamma}, && \gamma = \frac{a}{2b+d}.
    \end{align*}

    If instead we have $b+ \frac{d}{2} = 1$, we get $ \displaystyle \mathbb{E} W_2\left( \mu_L, \nu_L \right) \lesssim \frac{\sqrt{\log L}}{L^\gamma}$.
\end{corollary}

As the behaviour of Weyl's law for Neumann eigenvalues in a compact set of the plane is the same as in the case of compact manifolds without boundary, Corollary \ref{cor:smoothing} is the same argument as \cite[Lemma 7]{arias2024equidistributionpointsharmonicensemble}, based on the ideas of \cite{borda2023riesz}. The proof is included below for the sake of completeness. 
\begin{proof}[Proof of Corollary \ref{cor:smoothing}]
    Let's start by adding the expectation to the inequality of Theorem \ref{thm:smoothing}. After noticing that $\sqrt{\mathbb{E} f^2}$ is an $L^2$ norm with respect to the first intensity, we can apply Jensen and the triangle inequalities to get
    \begin{align}\label{eq:smoothingExpectation}
        \mathbb{E} W_2(\mu_L, \nu_L) \leq \sqrt{\mathbb{E} W_2^2(\mu_L, \nu_L)} \leq C_1 (dt)^{\frac{1}{2}} + \frac{2}{C_2(\nu_L, t)} \left( \sum_{k\geq 1} \frac{e^{-\lambda_k t}}{\lambda_k} \mathbb{E} | \widehat{\mu}_L(k) - \widehat{\nu}_L(k) |^2 \right)^{\frac{1}{2}}.
    \end{align}
    
    Now we claim that $\displaystyle \sum_{m=1}^{\infty} e^{\lambda_m t} \lambda_m^A \lesssim t^{- \left( A + \frac{d}{2} \right)}$ for $A > - \frac{d}{2}$. 

    To compute the series we define the following function and bound it using Weyl's law:
    \begin{align*}
        F(x) = \left|\left\{ m\in \mathbb{N} \;:\; \lambda_m \leq x \right\}\right| = C_d V_M x^{\frac{d}{2}} + O(x^{ \frac{d-1}{2}}) \lesssim x^{\frac{d}{2}}.
    \end{align*}
    Then the series can be bound as
    {\small
    \begin{align*}
        &\sum_{m=1}^{\infty} e^{-\lambda_m t} \lambda_m^A  = \int_0^\infty \hspace{-4pt} e^{-xt} x^A \;dF(x) = - \int_0^\infty\hspace{-3pt} F(x) \left( e^{-tx} x^A \right)' dx  = \int_0^\infty\hspace{-3pt} F(x) \left( t e^{-tx}x^A - e^{-tx} A x^{A-1} \right) dx, \\
        &I_1  = \int_0^\infty \hspace{-4pt} F(x) t e^{-tx} x^A dx \lesssim t \int_0^\infty \hspace{-4pt} e^{-tx} x^{A + \frac{d}{2}} \;dx = \int_0^\infty e^{-y} t^{- \left( A + \frac{d}{2}\right)} y^{A+\frac{d}{2}} \;dy = t^{-\left( A + \frac{d}{2} \right)} \Gamma\left( A+\frac{d}{2}+1 \right), \\
        & I_2 = \int_0^\infty F(x) e^{-tx} x^{A-1} \;dx \lesssim t^{-\left( A+\frac{d}{2} \right)} \int_0^\infty e^{-y} y^{A-1+\frac{d}{2}} \leq t^{-\left( A + \frac{d}{2} \right)} \Gamma\left(  A+ \frac{d}{2} \right).
    \end{align*}}
    This can be applied to the series from \eqref{eq:smoothingExpectation} with $A=b-1$. As $C_2$ is bounded below by assumption, we have
    \begin{align*}
        \mathbb{E}W_2\left( \mu_L, \nu_L \right) & 
        \lesssim t^{\frac{1}{2}} + \left( L^{-a} t^{-\left( b-1+\frac{d}{2} \right)} \right)^\frac{1}{2} 
        \underset{t=L^{-\alpha}}{=}  (L^{-\alpha})^\frac{1}{2} + \left( L^{-a+ \alpha \left( b-1+\frac{d}{2} \right)} \right)^\frac{1}{2}.
    \end{align*}
	and, optimizing the exponent in terms of $\alpha$ we get the exponent $\gamma = \frac{a}{2b+d}$.
	
	For the case when $b+ \frac{d}{2} = 1$, a little more care is needed, as the Gamma integrals do not converge. The first observation is that, because $F(x)=0$ for $0 \leq x \leq \lambda_1$ we can eliminate this part of the integrals and bound them with 
    \begin{align*}
        \sum_{m=1}^\infty e^{-t \lambda_m} \lambda_m^{b-1} &  \lesssim \int_{\lambda_1}^{\infty}F(x) e^{-tx} \left( t x^{b-1} + x^{b-2} \right) dx \lesssim \int_{\lambda_1}^{\infty} x^{\frac{d}{2}} e^{-tx} \left( t x^{b-1} + x^{b-2} \right) dx \\
        & = \int_{\lambda_1}^\infty t e^{-tx} + e^{-tx} x^{-1} dx = \int_{\lambda_1}^\infty e^{-y} (1+y^{-1}) dy \leq 1-\log \lambda_1 t \lesssim 1 - \log t.
    \end{align*}
	We can input this in the smoothing inequality, and after substituting $t=L^{-\alpha}$, we get the bound
    \begin{align*}
        \mathbb{E} W_2\left( \mu_L, \nu_L \right) \lesssim t^{\frac{1}{2}} + L^{ \frac{-a}{2}} \left( 1- \log t \right)^{\frac{1}{2}} = L^{\frac{-\alpha}{2}} + L^{ \frac{-a}{2}} \left( 1+ \alpha \log L \right)^{\frac{1}{2}} \lesssim L^{ \frac{-a}{2}} \sqrt{\log L}.
    \end{align*}
    which is the one appearing in the statement.
\end{proof}

\subsection{Fractional Sobolev spaces}
Let the fractional Sobolev space $ H^s(\Omega) = \left\{ f\in L^2(\Omega) \;:\; \norm{f}_{H^s(\Omega)} < \infty \right\} $ be characterized by the seminorm
\begin{align*}
	\norm{f}^2_{H^s(\Omega)} = \int |\widehat{f}(\xi)|^2 |\xi|^{2s} d \xi, \qquad s> 0.
\end{align*}
We can also consider the equivalent Gagliardo seminorm:
\begin{align*}
	\norm{f}_{\dot{H}^s(\Omega)}^2 = \iint_{\Omega \times \Omega} \frac{\left| f(x) - f(y) \right|^2 }{|x-y|^{2+ 2s}} dxdy \approx \norm{f}^2_{H^s(\Omega)}, \qquad 0 < s < 1. 
\end{align*}
 Our focus is the study of functions smooth on a convex set $ \Omega \subset \mathbb{R}^2 $ (and hence Lipschitz), that are extended by $ 0 $ outside. It is well-known that they are not in the Sobolev space $ H^1(\mathbb{R}^2) $, as the sharp jump makes the derivative not a function, but rather a measure. Another lesser known fact is that they are in the fractional spaces $ H^s(\mathbb{R}^2) $ when $ s < \frac{1}{2} $. 
 \begin{lemma}\label{lemma:fracSobolev}
	 Let $ \Omega $ be a Lipschitz domain and $ 0 < s < \frac{1}{2}  $. For any function $ f\in C^\infty(\Omega) \cap L^1(\mathbb{R}^2) \cap L^2(\mathbb{R}^2)$ supported on $ \overline{\Omega} $ we have $ \norm{f}_{\dot{H}^s(\mathbb{R}^2)} \leq C \norm{f}_{H^s(\Omega)}  $, with a constant $ C $ depending on $ s $ and $ \Omega $ but not $ f $.
 \end{lemma}
 \begin{proof}
	 This lemma in this exact form is a consequence of \cite[Theorem 3.32]{mclean_2000}, where the $ L^1 \cap L^2 $ hypothesis is used, and the equivalence between the fractional Sobolev seminorms. See also \cite[Theorems 3.33 and 3.40]{mclean_2000}.
 \end{proof}
 This result tells us that, when considering the eigenfunctions $ \phi_k $, it might be useful to obtain fractional Sobolev seminorms, as they will be finite. For this, notice that these seminorms can be bounded in terms of the eigenvalue.
\begin{lemma}
	Let $ \phi $ be a Neumann eigenfunction with eigenvalue $ \lambda $. Then we have $ \norm{f}^2_{\dot{H}^{s}(\Omega)} \lesssim \lambda^s $, with the constant depending on $ \Omega $ only.
\end{lemma}
\begin{proof}
	Up until now we have considered only seminorms, but we will need to consider now norms by adding the $ L^2 $ term:
	\begin{align*}
		\norm{f}_{W^{s, 2}(\Omega)} = \norm{f}_{L^2(\Omega)} + \norm{f}_{\dot{H}^s(\Omega)}.
	\end{align*}
	The lemma follows from two simple observations. First, for the eigenfunctions, as $ \norm{\phi}_{L^2} = 1 \lesssim \lambda $, the seminorm and norm are comparable. Secondly, we can interpolate using the Gagliardo-Nirenberg-Sobolev inequality. Following \cite[Theorem 1]{BREZIS20181355}, one has
	\begin{equation*}
		\norm{f}^2_{W^{s, 2}(\Omega)} \lesssim \norm{f}^{2- 2s}_{L^2(\Omega)} \norm{f}^{2s}_{W^{1, 2}(\Omega)}.
	\end{equation*}
	Finally, $ \norm{f}_{W^{1, 2}(\Omega)} = 1 + \lambda^{ \frac{1}{2} } \approx \lambda^{ \frac{1}{2} }  $.
\end{proof}

Combining both results, what we have seen is that when $ 0 < s <\frac{1}{2}  $ and $ \phi $ is an eigenfunction,
\begin{align}\label{eq:sobolev}
	\int \left| \widehat{\phi}(\xi)  \right|^2 |\xi|^{2s} d\xi \lesssim \lambda^s. 
\end{align}

\subsection{Homogeneous Point Processes}

\begin{proof}[Proof of Theorem \ref{thm:generalthm}:] The objective is to apply the smoothing procedure via Corollary \ref{cor:smoothing} for
    \begin{align*}
        \mu_L = \frac{1}{N} \sum_{x \in \Xi_L \cap \Omega} \delta_{x}, && \nu_L = \nu = \frac{1}{|\Omega|} dx|_{\Omega}, && N = \# \Xi_L \cap \Omega.
    \end{align*}
    Due to $N$ being a random variable and potentially really small, we will consider two cases: the (unlikely) one where there are very few points, and the (likely) one where there are enough points to apply the smoothing successfully,
    \begin{align*}
        \mathbb{E} W_2(\mu_L, \nu) = \mathbb{E} \left[ W_2(\mu_L, \nu) \carac_{\left\{ N < \frac{1}{2} \mathbb{E}N \right\}} \right] + \mathbb{E} \left[ W_2(\mu_L, \nu) \carac_{\left\{ N \geq \frac{1}{2} \mathbb{E}N \right\}} \right].
    \end{align*} 
    \paragraph{Case 1:} Assume $N < \frac{1}{2} \mathbb{E} N$. Then one can estimate the probability of this event as 
    \begin{align*}
        \mathbb{P} \left( N< \frac{1}{2} \mathbb{E} N \right) \leq \mathbb{P} \left( \left|N - \mathbb{E} N  \right| >   \frac{C}{2}L |\Omega| \right) \leq \frac{\Var N}{\left( \frac{C}{2}L |\Omega|  \right)^2}   \lesssim \frac{L \norm{\carac_\Omega}^2_{L^2}}{\left( \frac{C}{2}L |\Omega|  \right)^2} \approx \frac{1}{L}.
    \end{align*}
    Therefore, as the Wasserstein distance is bounded by, let's say $W_2 (\mu_L, \nu) \leq \operatorname{diam}(\Omega)$, we get 
    \begin{align*}
        \mathbb{E} \left[ W_2(\mu_L, \nu) \carac_{\left\{ N < \frac{1}{2} \mathbb{E}N \right\}} \right] \lesssim \frac{1}{\sqrt{L}},
    \end{align*}
    which is a smaller asymptotic than the final one.
    \paragraph{Case 2:} Assume $N \geq \frac{1}{2} \mathbb{E} N = L \frac{C|\Omega|}{2}$. We can apply the smoothing procedure and hence reduce the problem to studying the statistic
    \begin{align*}
        \mathbb{E} \left| \frac{1}{N} \sum_{x\in \Xi_L \cap \Omega} \phi_k(x) - \nu(\phi_k) \right|^2 & = \mathbb{E} \left| \frac{1}{N} \sum_{x\in \Xi_L \cap \Omega} \phi_k(x) \right|^2 \hspace{-0.4em} \lesssim \frac{1}{L^2} \mathbb{E} \left| \sum_{x\in \Xi_L \cap \Omega} \phi_k(x) \right|^2 \\
        &= \frac{1}{L^2} \Var \left(  \sum_{x\in \Xi_L \cap \Omega} \phi_k(x) \right) \lesssim \frac{1}{L}.
    \end{align*}  
    After putting this in Corollary \ref{cor:smoothing} we get the statement's bound.
\end{proof}

\begin{proof}[Proof of Corollary \ref{thm:homogeneous}]
    To apply Theorem \ref{thm:generalthm}, it is only needed to prove the variance inequality \eqref{eq:varinequality}, which is an easy consequence of the reproductive property of the kernel and the variance formula:
    \begin{align*}
        \Var \left( \sum_{x \in \Xi_L} f(x) \right) = \frac{1}{2}\iint \left|f(x) - f(y)\right|^2 |K_L(x, y)|^2 \lesssim  \int |f(x)|^2 K(x, x) \approx L \norm{f}^2_{L^2}. 
    \end{align*}
\end{proof}

\begin{proof}[Proof of Theorem \ref{thm:hyperuniform}]
	One can follow the general scheme presented in the above proof of Theorem \ref{thm:generalthm}, and just modify the bound in ``Case 2''. Indeed, we were using that 
	\begin{align*}
		\Var \left( \sum_{x \in \Xi_L \cap \Omega} \phi_k(x)  \right) \lesssim L.  
	\end{align*}
	But for hyperuniform point processes one can be more precise. Let's just denote by $ \phi_k $ the extension by $ 0 $ outside of $ \overline{\Omega} $, so that we just have the variance 
	\begin{align*}
		\Var \left( \sum_{x \in \Xi_L} \phi_k(x) \right) = L \int \left| \widehat{\phi}_k(\xi) \right|^2 s \left( \frac{|\xi|}{\sqrt{L}}  \right)  d \xi.
	\end{align*}
	For point processes of Type $ III $, we know that $ s(r) \approx \min (r^{\alpha}, 1) \lesssim r^\alpha $ for some $0 < \alpha < 1 $. This means that using \eqref{eq:sobolev}
	\begin{align*}
		\Var \left( \sum_{x \in \Xi_L} \phi_k(x) \right) \lesssim L^{1 - \frac{\alpha}{2} } \int \left| \widehat{\phi}_k(\xi) \right|^2 |\xi|^{\alpha} d \xi \leq  L^{1 - \frac{\alpha}{2} } \lambda^{ \frac{\alpha}{2} }.
	\end{align*}
	This inequality combined with Corollary \ref{cor:smoothing} gives the $ \frac{1}{\sqrt{L}}  $ bound.

	Notice that for point processes of Type $ I $ and $ II $, although $ \alpha \geq 1 $, we still have a bound of the type $ s(r) \lesssim r^{ \frac{1}{2} } $, so we can apply the previous argument. 
	
\end{proof}

\subsection{Random normal matrices}

\subsubsection{Previous results} 

The random normal matrices model is a topic of research, so its properties are well studied. Notably, the behaviour of the kernel $ K $ has been determined in different regimes. Several results from the literature are utilized and adapted to our setting below, capturing the main properties that we will need from the kernel. Recall that our regularities assumptions are summarized in Remark \ref{remark:regularity}.

The fact the eigenvalues tend to accumulate on the droplet $S$ is well illustrated in Figure \ref{fig:rnm}, but it can also be seen by an exponential decay outside of it. For instance, thanks to \cite[Proposition 3.6]{BerezinTransform} there is a constant $C>0$ such that
\begin{align*}
    K(z, z) \leq C n e^{- n  (Q(z) - \widehat{Q}(z))}.
\end{align*} 
Notice that because of the growth of $Q$ and $\widehat{Q}$, for $|z|$ big enough one has $Q(z) - \widehat{Q}(z) \geq C_0 \log |z|^2$. This means that for let's say $U_0 = \left\{ |z| > R_0 \right\}$ it can be bounded
\begin{align*}
    K(z, z) \lesssim n e^{-n C \log (2 + |z|^2)} = \frac{n}{(2+|z|^2)^{nC}} \lesssim \frac{1}{(2+|z|^2)^3}, && z\in U_0.
\end{align*}
On the other hand, when $z$ is close to $\partial S$, using \cite[Proposition 2.5]{WardIdentities} one has the lower bound $Q(z) - \widehat{Q}(z) \geq C_1 d(z, \partial S)^2$ for $z\in U_1 \backslash S$ and $U_1$ a bounded neighbourhood of $S$. We are going to further assume there is a bit of separation with respect to the droplet, that is, $d(z, \partial S) \geq a_0 \frac{\sqrt{\log n}}{\sqrt{n}}$ for some fixed constant $a_0\in \mathbb{R}_+$ that will be appropriately chosen below. Combining both facts,
\begin{align*}
    K(z, z) \lesssim n e^{-n C_1 a_0^2 \frac{\log n}{n}} = n^{1 - C_1 a_0^2}, && z \in U_1 \backslash S.
\end{align*}
So, if we take $a_0 \geq \frac{1}{\sqrt{C_1}}$, the kernel can be controlled by $K(z, z) \lesssim (2+|z|^2)^{-3}$ for $z\in U_1 \backslash S$ such that $d(z, \partial S) \geq a_0 \frac{\sqrt{\log n}}{\sqrt{n}}$. In the region that is left, that is, $\mathbb{C} \backslash (U_0 \cup U_1) $, it can be used that the functions $Q(z) - \widehat{Q}(z)$ are bounded because it is a compact set. All in all, the decay we need has the form
\begin{align}\label{eq:exteriorbound}
    K(z, z) \lesssim \frac{1}{(2+|z|^2)^3}, && \forall z \notin S, \;\; d(z, \partial S) \geq a_0 \frac{\sqrt{\log n}}{\sqrt{n}}.
\end{align}

Another relevant property of the kernel is that it is uniformly approximated by the equilibrium as long as we are in the interior of the droplet a bit far away from the boundary. 

Let $ \psi $ be the unique analytic function defined in a neighbourhood of $ (z, \overline{z}) $ such that $ \psi(z, \overline{z}) = Q(z) $. Let the approximation of the kernel be
\begin{align*}
	K^\#_n (z, w) = n \partial_1 \partial_2 \psi (z, \overline{w}) e^{ \frac{n}{2} (2 \psi(z, \overline{w}) - Q(z) - Q(w)) }.
\end{align*}
Notice that $ K^\#_n(z, z) = n \frac{1}{4} \Delta Q(z)  $ is proportional to the equilibrium measure. As consequence of \cite[Theorem 2.1]{ameur_2011}, there are constants $ a_0, a_1 > 0 $ such that if $ z\in S $ with $ d(z, \partial S) > a_0 \frac{\sqrt{\log n}}{\sqrt{n}} $ then
\begin{align}\label{eq:bulkbound}
	\left| K_n(z, w) - K_n^\#(z, w) \right| \leq C, && |z-w| < a_1  \frac{\sqrt{\log n}}{\sqrt{n}}.
\end{align}

The constant $a_0$ is taken to be the same as in \eqref{eq:exteriorbound}, at the cost of possibly worsening one of them. 

For the case when $ z $ and $ w $ are not close, \cite[Corollary 8.2 ]{BerezinTransform} imply that there exists some $ \varepsilon>0 $ such that for any $ z\in S $,
\begin{equation} \label{eq:distantzw}
	\left| K_n(z, w) \right|^2 \leq C n^2 e^{- \varepsilon \sqrt{n} \min \left\{ \frac{1}{8} d(z, S^c), |z-w| \right\} }.
\end{equation}
Finally, for the behaviour in the edge we have the following result. Denote by $ d(z) $ the signed distance to the boundary of $ S $, so that it is negative on $ \mathring{S} $ and positive on $ \mathbb{C} \backslash S $. Then, for every fixed $ C_1 > 0 $ there are constants $ c, C $ such that for $ n $ big enough and $ |d(z)| \leq C_1 \frac{\sqrt{\log n}}{\sqrt{n}}  $,
\begin{equation}\label{eq:edge}
	\left| \frac{1}{n} K(z, z) - \frac{1}{4} \Delta Q(z) \carac_{S}(z)    \right| \leq C \left( e^{-c n d(z)^2} + \frac{\log^3 n}{\sqrt{n}}  \right).  
\end{equation}
Indeed, for $ n $ big enough we can write $ z = z_0 + \vec{n}(z_0) d(z) $ for $ z_0 \in \partial S $ and $ \vec{n}(z_0) $ is the normal outward unit vector on $ \partial S $ at $ z_0 $. Using \cite[Lemma 1.3]{MarzoUniversality}, 
\begin{align*}
	\left| \frac{1}{n} K(z, z)  \right| & = \frac{1}{4} \Delta Q(z_0) \left( \frac{1}{2} + O\left( \frac{\log^3 n }{\sqrt{n}}  \right)   \right) \left| \operatorname{erfc}\left( \sqrt{2} \xi \right)  \right|, 
\end{align*}
for the number $ \xi = d(z) \sqrt{n \frac{1}{4} \Delta Q(z_0) } = O(\sqrt{\log n}) $. The complementary error function is normalized as 
\begin{equation*}
	\operatorname{erfc}(s) = \frac{2}{\sqrt{\pi}} \int_{s}^\infty e^{-z^2} dz. 
\end{equation*}
Observe it satisfies $ \left| \frac{1}{2} \operatorname{erfc(s)} - \carac_{(-\infty, 0)}(s)  \right| = \frac{1}{2} \operatorname{erfc}(|s|) \leq e^{-s^2}   $. From this it follows that
{\small
\begin{align*}
	\left| \textstyle\frac{1}{n}K(z, z) - \frac{\Delta Q(z)}{4}  \carac_S(z)   \right| & \leq \frac{\textstyle\Delta Q(z_0)}{4}  \left( \frac{1}{2} \operatorname{erfc}(\sqrt{2}\xi) - \carac_S(z)  \right)     + O\left( \textstyle\frac{\log^3 n}{\sqrt{n}}  \right) +  \frac{\carac_S(z)}{4} \left( \Delta Q(z_0) - \Delta Q(z) \right) \\
	& \lesssim e^{-c n d(z)^2} + \frac{\log^3 n}{\sqrt{n}}.  
\end{align*}}
\subsubsection{Proof of Theorem \ref{thm:RNM}}

\begin{remark}
    Doing a change of variables for a different norm than the Euclidean $|\cdot|$ one is a technical detail that we will recall before we begin the proofs. In the standard polar coordinates one writes $z \mapsto (r, \theta)$ to represent the change $z= r e^{i \theta}$. This way $r$ is $|z|$ and $e^{i \theta}$ represents the point in $\partial \mathbb{D}$ with same angle as $z$. We are going to a similar decomposition $z = r u(\theta)$, but now imposing $u(\theta)\in \partial S$ instead of the circle. 
    
    Assume without loss of generality that $0\in \mathring{S}$. The Minkowski functional is defined for $z \neq 0$ as
    \begin{align*}
        p_S(z) = \inf \left\{ s > 0 \;:\; \frac{x}{s} \in S \right\}.
    \end{align*} 
    Because $0\in \mathring{S}$, it is a standard property that $p_S$ is a norm. Moreover, because $S$ is a compact set with smooth boundary, $\frac{z}{p_S(z)} \in \partial S$ and we can write $z = p_S(z) u(\theta)$ as our change of variables:
    \begin{align*}
        dz = J(r, \theta) dr d \theta, && J(r, \theta) = \left|\det \left( \frac{\partial z}{\partial r \partial \theta} \right)\right| = r |\det(u(\theta), u'(\theta))|.
    \end{align*}   
    Notice that $u(\theta)$ can we written as $u(\theta) = \rho(\theta) e^{i \theta}$, where $\rho(\theta) \in \mathbb{R}_+$. This allows to compute the determinant as $\det(u(\theta), u'(\theta)) = (\rho(\theta))^2$. It is clear from the definition of $\rho$ that it can be bounded by
    \begin{align*}
        0 < \min \left\{ |z| \;:\; z\in \partial S \right\} \leq \rho(\theta) \leq \max \left\{ |z| \;:\; z\in \partial S \right\}.
    \end{align*}
    So, all in all the change of variables has Jacobian $r \rho(\theta)^2 dr d \theta$ with $\rho$ a bounded function.    
\end{remark}

\begin{proof}[Proof of Theorem \ref{thm:RNM}]

    Let $z_1, z_2, \dots, z_n$ be the random points given by a RNM model with the regularity assumptions of the statement. Recall that the empirical measure is $\mu = \frac{1}{n} \sum \delta_{z_i}$ and the limit is $d\sigma = \frac{1}{4} \Delta Q \carac_{S} dA $. 

    First, we can assume without loss of generality that $0\in \mathring{S}$. The complex plane will be divided into three regions so that the points in each one are better studied separately. Namely, we will consider the points far away from the droplet, the edge of the droplet, and the bulk. The bulk, the edge, and the inside of the edge are defined as follows
    \begin{align*}
        S_n = r_n \mathring{S}, && E_n = (r_n' S) \backslash S_n, && E = E_n \cap \mathring{S}.
    \end{align*}
    \begin{figure}[ht]
        \centering
        \includegraphics[width=0.5\linewidth]{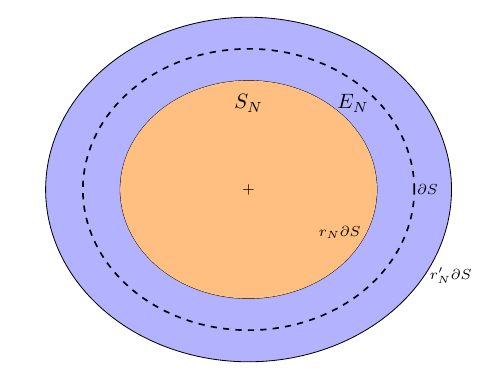}
        \caption{Representation of the edge (in blue) and the bulk (in orange), and the relevant boundaries.}
    \end{figure}

    The idea is that we can control the kernel both inside and outside the droplet, thanks to \eqref{eq:exteriorbound} and \eqref{eq:bulkbound}, provided there is enough separation from the boundary. Then the edge $E_n$ must contain the remaining part. The contraction factor $r_n$ is chosen this way: let $\rho = d(0, \partial S)$, then by convexity we have for any $0 < r < 1$ the inclusion $r S + (1-r)B(0, \rho) \subseteq S$. This means that $d(rS, \partial S) \geq (1-r)\rho$. Let's abbreviate $\delta_n = \frac{\sqrt{\log n}}{\sqrt{n}}$. To guarantee the inclusion $S_n = r_n \mathring{S} \subseteq \left\{ z\in S \;:\; d(z, \partial S) > 2a_0 \delta_n \right\}$ we will choose $r_n = 1 - \frac{2a_0}{\rho} \delta_n$ for $n$ big enough.  

    To define $r_n'$, first notice that because all norms in $\mathbb{R}^2$ are equivalent, $|\cdot| \geq C_0 p_S(\cdot)$. For any $z\notin S$ let $w\in \partial S$ be such that $d(z, \partial S) = |z-w|$. This implies
    \begin{align*}
        d(z, \partial S) = |z-w| \geq C_0 p_S(z-w) \geq C_0(p_S(z) - p_S(w)) = C_0(p_S(z) - 1).
    \end{align*}
    So, if we have $z\notin r_n' S$ and we want to guarantee $d(z, \partial S) > a_0 \delta_n$, it is enough to take $r_n' = 1 + \frac{2a_0}{C_0} \delta_n$.   

    This manuscript contribution is the application of the heat smoothing to the bulk, where most of the points are. The intermediate measures considered are:
    \begin{align*}
        \mu_1 =& \mu(\mathbb{C}\backslash r_n'S) \widehat{\lambda} + \mu|_{E_n} + \mu|_{S_n}, \\
        \mu_2 =& \mu(\mathbb{C}\backslash S_n) \widehat{\lambda} + \mu|_{ S_n}, \\
        \mu_3 =& \mu(\mathbb{C}\backslash S_n) \widehat{\lambda} + \mu(S_n) \frac{\sigma |_{S_n}}{\sigma (S_n)}.
    \end{align*}
    Here $ \widehat{\lambda} $ is an absolutely continuous probability measure supported on $ E $. Its density is comparable to the one $ \frac{\sigma |_{E}}{\sigma (E)} $, and the precise form it takes is explained below in \eqref{eq:lambda}. 

    Hence each of the following terms will be bounded separately
    \begin{align*}
        \mathbb{E} W_2 (\mu, \sigma) \leq \mathbb{E} W_2 (\mu, \mu_1) + \mathbb{E} W_2 (\mu_1, \mu_2) + \mathbb{E} W_2 (\mu_2, \mu_3) + \mathbb{E} W_2 (\mu_3, \sigma)
    \end{align*}
    and the theorem is a consequence of Lemmas \ref{lemma:RNM1}, \ref{lemma:RNM2}, \ref{lemma:RNM3} and \ref{lemma:RNM4} which study the corresponding terms above.
\end{proof}

Using equation~\eqref{eq:exteriorbound} and the change of variables, we can bound the expected number of points outside of $r_n' S$: 
\begin{align*}
    \mathbb{E} \# \left\{ z \notin r_n'S \right\} & = \int_{\mathbb{C}\backslash r_n'S} K(z, z) dz \lesssim \int_{\mathbb{C}\backslash r_n'S} \frac{1}{(2+|z|^2)^3} dz \lesssim 1.
\end{align*}
This approach works to bound the expected Wasserstein distance outside of the droplet as well.

\begin{lemma}\label{lemma:RNM1}
    The points far away from the droplet satisfy $\mathbb{E} W_2 (\mu, \mu_1) \lesssim \frac{1}{\sqrt{n}}$.
\end{lemma}

\begin{proof}
    The main tool is equation \eqref{eq:exteriorbound} combined with the inequality $d(z, \partial S) \gtrsim p_S(z) - 1$ explained before. This allows us to bound the Wasserstein distance by the transport map
    \begin{itemize}
        \item For the points $z_i \in r_n'S$, there is no need to do anything.
        \item For the points $z_i \notin r_n'S$, send all the mass uniformly into $\widehat{\lambda}$. Then each point moves a mass of $\frac{1}{n}$ and the distance can be bounded by $ \leq |z_i| + \operatorname{diam}(S) \lesssim |z_i|$.
    \end{itemize}
    So, $W_2^2(\mu, \mu_1) \leq \frac{1}{n} \sum_{z_i \notin r_n'S} |z_i|^2$ and thus
    \begin{align*}
        \mathbb{E}W_2^2(\mu, \mu_1) & \lesssim \frac{1}{n} \int_{z\notin r_n'S} |z|^2 K(z, z) dz \lesssim  \frac{1}{n}\int_{z\notin r_n'S} \frac{|z|^2}{(1+|z|^2)^3} \lesssim \frac{1}{n}.
    \end{align*}
    Finally, $\mathbb{E} W_2(\mu, \mu_1) \lesssim \frac{1}{\sqrt{n}}$.

\end{proof}

As we are working with the normalized Lebesgue measure, we will write $|B| = \int_B dA(z) = \frac{1}{\pi} \Vol(B)$. We will also denote the first intensity $ \rho(z) = \frac{1}{n\pi} K(z, z)  $. Note that it is the density of $ \mathbb{E} \mu $.

\begin{lemma}\label{lemma:RNM2}
    The points in the edge satisfy $\mathbb{E} W_2 (\mu_1, \mu_2) \lesssim \frac{1}{\sqrt{n}}$.
\end{lemma}
The proof for the edge is a generalization of the one done by Prod'Homme \cite{prodhomme} for the Ginibre ensemble. The main idea is to consider the following series of intermediate measures. For any $k = 0, 1, \dots$
\begin{itemize}
    \item $\{B_j^k\}_{j=1}^{2^k}$ is a partition of $E_n$.
    \item Each set has the form $B_j^k = \left\{ z\in E_n \;:\; \arg z \in I_j^k\right\}$ where $I_j^k$ is an interval of the form $[a,b)$.
    \item They are constructed recursively, this is, for any $I_j^k$ we split it in $I_j^k = I_{i}^{k+1} \dot{\cup} I_{i+1}^{k+1}$ so that we have $\frac{1}{2} |B_j^k| = | B_{i}^{k+1} | = |B_{i+1}^{k+1}|$.
    \item We denote the part that is inside the bulk by $\tilde{B}_j^k = B_j^k \cap S$.      
\end{itemize}

We are going to consider a tubular neighbourhood of the boundary $ \partial S $, so that we have the coordinates $ z = \zeta(s) + t \vec{n}(s) $ with $ \zeta(s) \in \partial S $, $ \vec{n}(s) $ the normal inwards vector at $ \zeta(s) $, and $ |t| $ small enough. The Jacobian is given by $ 1 + \kappa(s) t $. For $ n $ big enough, we have that the edge $ E_n $ is contained in this tubular region, so that we can parametrize $ E $ with this coordinates. In this case, each region $ B^k_j $ is a union of fibers $ \{s\}\times (t^-(s), t^+(s)) $. Let's denote by $ F_s = \left\{ s \right\} \times (0, \tau(s)) $ the part of the fiber that is inside the droplet. The length is $ \tau(s) \approx \delta_n $ and depends smoothly on $ s $. Notice that the Jacobian is of the order $ 1 + O(\delta_n) $, as $ t^\pm(s) \approx \delta_n $. For a region $ B= B_j^k $ we define the measure
\begin{align}\label{eq:lambda}
	\lambda_B := m(s) ds \operatorname{Unif}_{F_s}(t) , && m(s):= \int_{t^-(s)}^{t^+(s)} \rho(z(s,t)) \left( 1 +\kappa(s) t \right) dt.
\end{align}
This is a measure supported on $ \tilde{B} = B \cap S_n$. They satisfy $ \lambda_B|_{B'} = \lambda_B'$. Informally, they are a redistribution of $ \rho_n $ along the fibers, as they have the same $ s $-marginal as $ \rho_n $. This formal step is need because we do not have lower bounds of the type $ \rho_n(z) \gtrsim dz $ on the edge. 

Their mass is $ \lambda_B(B) = \rho_n(B) $. We denote its normalization as $ \widehat{\lambda}_B = \frac{1}{\rho_n(B)} \lambda_B  $. We also abbreviate $ \widehat{\lambda}_E = \widehat{\lambda} $. Consider then the (random) measures $\displaystyle \nu^k = \sum_{j=1}^{2^k} \mu(B_j^k) \widehat{\lambda}_{B^k_j} $. We will let $k=0, 1, \dots, \mathbf{K}$ up until it is satisfied $2^{-\mathbf{K}} \in [\delta_n, 2 \delta_n)$.

The proof requires several lemmas that are explained below. Some of them are a generalization of the ones appearing in \cite{prodhomme} for the Ginibre ensemble, valid now for general potentials. 
\begin{lemma}\label{lemma:bounds_expectation}
    Given a measurable set $B \subseteq S$ with $z\in B \Rightarrow d(z, \partial S) > 2a_0 \delta_n$ then we have
    \begin{align*}
        \mathbb{E} \left( \# \left\{ z_i \in B \right\} \right) \approx n |B|, && \mathbb{E} ( \mu(B)) \approx |B|.
    \end{align*}
\end{lemma}
\begin{proof}
    On one hand,
    \begin{align*}
        \mathbb{E} \left( \# \left\{ z_i \in B \right\} \right) & = \int_{B} K(z, z) dA(z) \lesssim n \int_B dA(z) = n |B|.
    \end{align*}
    On the other, using that we are in the bulk and can apply \eqref{eq:bulkbound}
    \begin{align*}
        \mathbb{E} \left( \# \left\{ z_i \in B \right\} \right) & = \hspace{-2pt}\int_B \frac{n}{4} \Delta Q(z) dA(z) + \hspace{-2pt}\int_B K(z, z) - K^{\#}(z, z) dA(z) \geq \frac{n}{4 \pi} C |B| - C'|B| \gtrsim n |B|.
    \end{align*}
\end{proof}

\begin{lemma}\label{lemma:variance}
    Given a measurable set $B \subseteq \mathbb{C}$ then
    \begin{align*}
        \Var \left( \# \left\{ z_i \in B \right\} \right) \leq \mathbb{E}  \# \left\{ z_i \in B \right\} \lesssim n |B|, && \Var ( \mu(B)) \leq \frac{\mathbb{E} \mu(B)}{n} \lesssim \frac{|B|}{n}.
    \end{align*}
\end{lemma}
\begin{proof}
    This is just an application of the reproducing property of the kernel:
    \begin{align*}
        \Var \left( \#\left\{ z_i \in B \right\} \right) & = \frac{1}{2} \iint \left|\carac_B(z) - \carac_B(w)\right|^2 |K(z,w)|^2 dA(z)dA(w) \\
        &\leq \iint \carac_B(z) |K(z, w)|^2 dA(z)dA(w) 
         = \int \carac_B K(z, z) dA(z) \lesssim n |B|.
    \end{align*}
\end{proof}

We are choosing to divide the partitions so that the Lebesgue measure is exactly divided in two. This choice does not affect the division of the angles or the mass of the equilibrium measure too much.

\begin{lemma}
    We have $|I_j^k| \approx \frac{1}{2^k}$ and $|B^k_j| \approx 2^{-k} \delta_n$.
\end{lemma}
\begin{proof}
    Using the change of variables explained at the beginning of the section, one may express the area as the integral $|B_j^k| = \int_{\theta \in I_j^k} \rho(\theta)^2 d \theta \int_{r_n}^{r_n'} r dr$, where $ \rho(\theta) \approx 1$. Hence,
    \begin{align*}
        \frac{|E|}{2^k} = |B_j^k| \lesssim (r_n'-r_n) r_n'  |I_k^j| \approx \delta_n |I^k_j| && \Longrightarrow && |I_j^k| \gtrsim \frac{|E|}{\delta_n 2^k} \approx \frac{1}{2^k} \\
        \frac{|E|}{2^k} = |B_j^k| \gtrsim (r_n'-r_n) r_n  |I_k^j| \approx \delta_n |I^k_j| && \Longrightarrow && |I_j^k| \lesssim \frac{|E|}{\delta_n 2^k} \approx \frac{1}{2^k}
    \end{align*}   
\end{proof}

We will often take the convention $B^k_j = B$, $ B_{i}^{k+1} = B_1$ and $ B_{i+1}^{k+1} = B_2$ when we are working with an arbitrary set of the partitions and its two subsets.

\begin{lemma}\label{lemma:densitym}
    Uniformly in $ s $ and $ n $, one has $ m(s) \approx \delta_n $. 
\end{lemma}
This lemma has several consequences that will be needed:
\begin{itemize}
	\item $ \rho_n(B) \approx |B| \approx 2^{-k} \delta_n $.
	\item The density of $ \widehat{\lambda}_B $ on $ B $ is $ \approx \frac{1}{\rho_n(B)} \frac{m(s)}{\tau(s)} \approx \frac{1}{|B|}    $.
	\item There is some constant $ 0 < c < 1 $ such that for all of regions $ \frac{ \rho_n(B_1)}{\rho_n(B)} \in [c, 1-c] $.
\end{itemize}
\begin{proof}
	The upper bound follows from $ K(z, z) \lesssim n $ and $ t^\pm(s) \approx \delta_n $. For the lower bound, notice we have left some part of the ``bulk'' on the edge region. Indeed, the edge contains the band $ \left\{ z \in S \;:\; a_0 \delta_n < d(z, \partial S) < 2a_0 \delta_n \right\} \subset E  $, where we can apply \eqref{eq:bulkbound} and hence $ K(z, z) \gtrsim n $. As in tubular coordinates this band corresponds with the part of the fiber $ a_0 \delta_n < t < 2a_0 \delta_n $, we have the other bound. 
\end{proof}

\begin{proof}[Proof of Lemma \ref{lemma:RNM2}]

It is enough to bound $W_2(\mu|_{E_n}, \mu(E_n) \widehat{\lambda} )$. As expected, we are going to use the intermediate measures constructed. 

\paragraph{Bounding the last measure:} Consider the transport map from $\nu^{\mathbf{K}}= \lambda$ to $\mu|_{E_n}$ that in each region $B_j^{\mathbf{K}}$ sends $\frac{1}{n} 
\nu|_{\tilde{B}_j^{\mathbf{K}}}$ to any of the points $z_i \in B_j^{\mathbf{K}}$. If $\mu(B_k^\mathbf{K})=0$ then there is no need to do anything. We are moving a mass of $\mu(B^{\mathbf{K}}_j)$ and the distance can be bounded by $\lesssim 2^{-\mathbf{K}} +  (r_n'-r_n) \lesssim \delta_n$, hence this gets 
\begin{align*}
    W_2^2(\nu^{\mathbf{K}}, \mu|_{E_n}) \lesssim \sum_{j=1}^{2^{\mathbf{K}}} \delta^2 \mu(B_j^{\mathbf{K}}) = \delta^2_n \mu(E_n).
\end{align*}

\paragraph{Bounding intermediate measures:} We will work individually in each region $B^k_j =: B$. Denote by $B_1, B_2$ its corresponding sets in the next measure $\nu^{k+1}$ and $\tilde{B}, \tilde{B}_1, \tilde{B}_2$ the intersections with $S$. What we will study then is $\mathbb{E} W_2^2 (\nu^k|_{\tilde{B}}, \nu^{k+1}|_{\tilde{B}})$.

We will distinguish two main cases. For this, fix a constant $C_0>0$ such that $\mathbb{E} \mu (B_i^k) \geq C_0 |{B}_i^k|$ for any of our regions, as seen in Lemma \ref{lemma:densitym}.

\paragraph{Case 1: } Assume there is some $i\in \left\{ 1, 2 \right\}$ such that $\mu(B_i) \leq \frac{C |{B}_i|}{2}$. The probability of this event is
\begin{align*}
    \mathbb{P} \left( \mu(B_i) \leq \frac{C_0 |B|}{2} \right) & = \mathbb{P} \left( \mu(B_i) - \mathbb{E} \mu(B_i) \leq C_0 |B_i| - \mathbb{E}\mu(B_i) - \frac{C_0 |B_i|}{2} \right) \\
    & \leq \mathbb{P} \left( \mu(B_i) - \mathbb{E} \mu(B_i) \leq - \frac{C_0 |B_i|}{2} \right) \\
    & \leq \mathbb{P}  \left( \left|\mu(B_i) - \mathbb{E}\mu(B_i) \right|  \geq \frac{C_0 |B_i|}{2} \right) \leq \frac{\Var \mu(B_i)}{\left(  \frac{C_0 |B_i|}{2}\right)^2} \lesssim \frac{1}{n |B_i|}.
\end{align*}
where at the end is used Lemma \ref{lemma:variance}. 
By the same token, for any $s\geq 2$
\begin{align*}
    \mathbb{P}\hspace{-2pt} \left( \mu(B) \geq s^2 C_0 |B| \right) \hspace{-2pt}\leq \hspace{-2pt}\mathbb{P} \hspace{-2pt}\left( \mu(B) - \mathbb{E} \mu(B) \geq (s^2 - 1) C_0 |B| \right) \hspace{-2pt}\leq \hspace{-2pt}\frac{\Var \mu (B)}{\left( (s^2-1) C_0 |B| \right)^2} \hspace{-2pt}\lesssim \hspace{-2pt}\frac{1}{n (s^2-1)^2 |B|}.
\end{align*}
To bound the transport map in the region $\tilde{B}$, notice that the distance can be bounded by $\lesssim \delta_n + 2^{-k} \lesssim 2^{-k}$. The mass is $\mu(B)$ naturally. Hence we have
\begin{align*}
    \mathbb{E} \bigg[ W^2_2(\nu^k|_{\tilde{B}}, \nu^{k+1}|_{\tilde{B}}) &\carac_{\mu (B_i)  \leq \frac{C_0 |B_i|}{2}}  \bigg] \lesssim 2^{-2k} \mathbb{E} \left[ \mu(B) \carac_{\mu (B_i) \leq \frac{C_0 |B_i|}{2}} \right].
\end{align*}

With respect to the expectation, we will differentiate the following cases:
\begin{itemize}
    \item $\mu(B) \leq 4 C_0 |B|$ 
    \item For $s\geq 2$,   $\;s^2 C_0 |B| < \mu(B) \leq (s+1)^2 C_0 |B|$  
\end{itemize}   
With this separation in mind,
{\small\begin{align*}
    \mathbb{E} \left[ \mu(B) \carac_{\mu (B_i) \leq \frac{C_0 |B_i|}{2}} \right]  & \leq \mathbb{E} \left[ \mu(B)  \carac_{\mu (B_i) \leq \frac{C_0 |B_i|}{2}} \carac_{\mu(B) < 4 C_0 |B|} \right] + \sum_{s=2}^\infty \mathbb{E} \left[ \mu(B)  \carac_{s^2 C_0 |B| < \mu(B) < (s+1)^2 C_0 |B|} \right] \\
    & \leq  4C_0|B| \mathbb{P}\left(\exists i \;:\; \mu (B_i) \leq \frac{C_0 |B_i|}{2}\right)  + \sum_{s=2}^\infty  (s+1)^2 C_0 |B| \mathbb{P} \left( s^2 C_0 |B| < \mu(B) \right) \\
    & \lesssim |B| \sum_{i=1, 2} \frac{1}{n |B_i|} + \hspace{-2pt}\sum_{s=2}^\infty (s+1)^2 |B| \frac{1}{n (s^2-1)^2 |B|} \lesssim \frac{1}{n} \left( \sum_{i=1, 2} 2 + \hspace{-2pt} \sum_{s=2}^\infty \frac{(s+1)^2}{(s^2-1)^2} \right).
\end{align*}}
All in all, 
\begin{align*}
    \mathbb{E} \bigg[ W^2_2(\nu^k|_B, \nu^{k+1}|_B) &\carac_{\mu (B_i)  \leq \frac{C |B_i|}{2}}  \bigg] \lesssim \frac{2^{-2k}}{n}.
\end{align*}
\paragraph{Case 2: } Assume $\mu(B_i) \geq \frac{C_0 |B_i|}{2}$ for both $i=1, 2$. Because the measures have the form
\begin{align*}
    \nu^k|_{\tilde{B}} = \mu(B) \frac{\lambda_B}{\rho(B)}, && \nu^{k+1}|_{\tilde{B}} = \mu(B_1) \frac{\lambda_{{B}_1}}{\rho({{B}_1})} + \mu(B_2) \frac{\lambda_{{B}_2}}{\rho({{B}_2})}.
\end{align*}  
The assumption implies that $\nu^{k+1}|_{\tilde{B}}\gtrsim \lambda_B \gtrsim dz|_{\tilde{B}}$, and we can apply equation \eqref{eq:dynamicWass}:
\begin{align*}
	W^2_2(\nu^k|_{\tilde{B}}, \nu^{k+1}|_{\tilde{B}}) & \lesssim  \operatorname{diam}(\tilde{B})^2 \int_{\tilde{B}} \left|\nu^{k+1}- \nu^k\right|^2 \lesssim 2^{-2k} \sum_{i=1, 2} \left| \frac{\mu(B)}{\rho(B)} - \frac{\mu(B_i)}{\rho(B_i)} \right|^2 \rho({{B}}_i) \\
	& \lesssim  2^{-2k}\sum_{i=1, 2} \rho({{B}}_i) \left| \frac{\mu(B)}{\rho(B)}- 1  \right|^2 +\rho({{B}}_i) \left| \frac{\mu(B_i)}{\rho(B_i)}- 1  \right|^2.
\end{align*} 

Let's focus our attention on one of the terms. When taking expectation, because $ \mathbb{E} \mu = \rho $, we get
\begin{align*}
	\rho({{B}}_i) \mathbb{E}\left| \frac{\mu(B)}{\rho(B)}- 1  \right|^2 = \frac{\rho(B_i)}{\rho(B)^2} \Var \left( \mu(B) \right) \lesssim \frac{1}{n}.   
\end{align*}
As $ \rho(B) \approx \rho(B_1) \approx \rho(B_2) $, the other terms are bounded identically. All in all, we have seen
\begin{align}
    \label{eq:case2}
    \mathbb{E} \bigg[ W^2_2(\nu^k|_{\tilde{B}}, \nu^{k+1}|_{\tilde{B}}) \carac_{\mu (B_i)  \geq \frac{C_0 |B_i|}{2} \forall i=1, 2}  \bigg] \lesssim \frac{2^{-2k}}{n}.
\end{align}

\paragraph{Combining intermediate steps: } In both cases, it has been proven that $\mathbb{E}  W^2_2(\nu^k|_{\tilde{B}}, \nu^{k+1}|_{\tilde{B}}) \lesssim \frac{2^{-2k} }{n}$. As there are $2^k$ regions ${\tilde{B}}$, this implies $\mathbb{E}  W^2_2(\nu^k, \nu^{k+1}) \lesssim \frac{2^{-k}}{n}$. Therefore the final bound is
\begin{align*}
    \sqrt{\mathbb{E}W_2^2 (\mu_1, \mu_2)} & \leq \sqrt{\mathbb{E}W_2^2 (\mu|_{E_n}, \nu^{\mathbf{K}})} + \sum_{k=0}^{\mathbf{K}-1} \sqrt{\mathbb{E}W_2^2 (\nu^k, \nu^{k+1})} \\
    &\lesssim \left( \delta_n^2 \mu(E) \right)^{\frac{1}{2}} + \sum_{k=0}^{\mathbf{K}-1} \frac{1}{\sqrt{n}} 2^{- \frac{k}{2}} \lesssim \frac{1}{\sqrt{n}}.
\end{align*}  

\end{proof}

Lemma \ref{lemma:variance} studies the variance of the number of points in an arbitrary set, and therefore is not always precise. When we are working with a set like the bulk, that has the form of the droplet minus a \textit{small part}, one should expect the variance to be small, as we expect the points that may or may not be in the set to be few (in comparison with the total). Using this idea and the reproducing property of the kernel we can get the following result.
\begin{lemma}\label{lemma:varedge}
    Given a measurable set contained in the bulk $B \subseteq S$ one has 
    \begin{align*}
        \Var \left( \#\left\{ z_i \in B \right\} \right) \lesssim 1 + n \left| r_n' S \backslash B\right|, && \Var \mu(B) \lesssim \frac{ \frac{1}{n}  + \left| r_n' S \backslash B\right|}{n}.
    \end{align*}
\end{lemma}
\begin{proof}
    Observe that as the total number of points $n$ is fixed one has $ \Var \left( \#\left\{ z_i \in B \right\} \right) = \Var \left(  \#\left\{ z_i \in \mathbb{C} \backslash B \right\}\right)$. Therefore
    \begin{align*}
        \Var \left(  \#\left\{ z_i \in \mathbb{C} \backslash B \right\}\right) & \leq \int_{\mathbb{C}\backslash B} K(z, z) dA(z) = \int_{\mathbb{C}\backslash r_n' S} K(z, z) + \int_{r_n' S \backslash B} K(z, z) \lesssim 1 + n |r_n' S \backslash B |.
    \end{align*} 
\end{proof}
This theorem will be applied to the set $B = S_n$, where it readily implies $\Var \mu (S_n) \lesssim \frac{|E|}{n} \approx \frac{\sqrt{\log n}}{n^{ \frac{3}{2} }} $.  

\begin{lemma}\label{lemma:RNM3}
	The points in the bulk satisfy $\mathbb{E} W_2 (\mu_2, \mu_3) \lesssim \frac{1}{\sqrt{n}}$.
\end{lemma}
\begin{proof}
    The key idea is to apply the heat smoothing, but there are two technical details to take into account. First, to guarantee that the implicit constants are the same and do not depend on the $S_n$ (which depends on $n$), we are going to rescale the measures and apply the smoothing in the same domain $S$. To the Wasserstein distance this is just multiplying by a power of the rescaling factor.  As this factor is $r_n \approx 1$, informally, this can be seen as applying the smoothing on the set $S_n$ and ignoring any possible dependence on $n$.

    Secondly, we need to guarantee there is some constant such that $ \mu(S_n)\sigma|_{S_n} \geq c dz|_{S_n}$. By the regularity assumptions, we already know $\sigma|_{S_n} \gtrsim dz|_{S_n}$, so it is only necessary to look at the mass. If it happened there was too little mass, we can just grossly bound the optimal map and take the expectation:
    \begin{align*}
        \mathbb{P} \left( \mu(S_n) < \frac{1}{2} \mathbb{E} \mu (S_n) \right) \leq \mathbb{P} \left( \left|\mu(S_n) - \mathbb{E} \mu (S_n) \right| > \frac{1}{2} \mathbb{E} \mu (S_n) \right) \leq \frac{\Var \mu (S_n)}{\left( \frac{1}{2}\mathbb{E} \mu (S_n) \right)^2} \lesssim \frac{1}{n}.
    \end{align*}

    On the other case, this is, if $ \mu(S_n) \geq \frac{1}{2} \mathbb{E} \mu (S_n)$, one can apply the heat smoothing so that the problem is reduced to estimating:
    \begin{align*}
        \mathbb{E} \left| \mu(\phi_k) - \frac{\mu(S_n)}{\sigma(S_n)} \sigma (\phi_k) \right|^2 \lesssim \mathbb{E} \left| \mu(\phi_k) - \mathbb{E} \mu (\phi_k)  \right|^2 + \mathbb{E} \left| \mathbb{E} \mu (\phi_k) - \frac{\mu(S_n)}{\sigma(S_n)} \sigma (\phi_k) \right|^2.
    \end{align*}

    The first term is the variance, which will need to be carefully bounded. For the other term, using equation \eqref{eq:bulkbound} and Lemma \ref{lemma:varedge}, it is easy to see that 
    \begin{align*}
        &\mathbb{E} \left| \mathbb{E} \mu(\phi_k) - \frac{\mu(S_n)}{\sigma(S_n)} \sigma(\phi_k) \right|^2  \\
        &\lesssim \left|\mathbb{E} \mu (\phi_k) - \sigma(\phi_k)\right|^2 + \frac{|\sigma(\phi_k)|^2}{\sigma(S_n)^{2}}   \left( \left|\sigma(S_n) - \mathbb{E} \mu (S_n)\right|^2 + \mathbb{E}\left|\mathbb{E} \mu (S_n) - \mu(S_n)\right|^2 \right)
    \end{align*}
    can be bounded by $\lesssim \frac{\sqrt{\log n}}{n^{ \frac{3}{2} }} \lesssim \frac{1}{n^{ 1 + s_0 }} $, for any $ s_0< \frac{1}{2}  $. For the final argument, any $0 < s_0 < \frac{1}{2} $ would work, so let's fix $ s_0 = \frac{1}{4} $ for instance. Still, I will keep the notation $ s_0 $, as the exponents will look clearer this way. Let's start now with the variance:
    \begin{align*}
        \Var \left( \mu(\phi_k) \right) = \frac{1}{2n^2} \iint \left| \phi_k(z) - \phi_k(w) \right|^2 \left| K(z,w) \right|^2 dzdw.   
    \end{align*}
    We are going to divide the integral into two regions:
	\begin{align*}
		R_1 &= \left\{ |z-w| < a_1 \delta_n; d(z, S^c) > 2 a_0 \delta_n \text{    or    } d(w, S^c) > 2 a_0 \delta_n \right\}, \\
		R_2 &= \left\{ |z-w| > a_1 \delta_n; d(z, S^c) > 2 a_0 \delta_n \text{    or    } d(w, S^c) > 2 a_0 \delta_n \right\} .
	\end{align*}
	Notice they do cover all $ (z, w) \in \operatorname{supp} \phi_k \times \operatorname{supp} \phi_k $, and that they are not necessarily disjoint. In $ R_1 $, we can use the approximation \eqref{eq:bulkbound},
	\begin{align*}
		\iint_{R_1} |\phi_k(z) - \phi_k(w)|^2 |K(z, w)|^2 \lesssim \iint_{\substack{d(z, S^c)> 2a_0 \delta_n \\ |z-w| < a_1 \delta_n}}|\phi_k(z) - \phi_k(w)|^2 \left( 1 + \left| K^\#_n(z, w) \right|^2  \right).
	\end{align*}
	The integral of the ``$ 1 $'' term can be bounded by $ \lesssim 1 $. For the second, if we were working with the Ginibre ensemble, that is, $ Q(z)= |z|^2 $, we would have that the lift of $ Q $ would be $ \psi(z, w) = zw $, and the approximation would be $ |K^\#_n(z, w)| = n e^{- \frac{n}{2} |z-w|^2 }  $. When the potential weight is not this one, we still have this asymptotic:
	\begin{equation*}
		\left| K^\#_n(z, w) \right|^2 = n^2 \left| \partial_1 \partial_2 \psi( z, \overline{w}) \right|^2 e^{n \Re \left\{ 2 \psi(z, \overline{w}) - Q(z) - Q(w) \right\} } \lesssim n^2 e^{-n |z-w|^2 \frac{1}{4} \Delta Q(w) }.  
	\end{equation*}
	This claim follows from bounding over a compact set of $ \partial_1 \partial_2 \psi $ and the Taylor expansion of the exponent
	\begin{equation*}
	    \Re\left\{ 2\psi(z, w) - Q(z) - Q(w) \right\} = -|z-w|^2 \frac{1}{4} \Delta Q(w) + O(|z-w|^3).
	\end{equation*}
	Now, because the function $ t \mapsto t^{ 1+s_0  } e^{-t}\leq 1$, we can bound the integral using fractional Sobolev norms. Namely, using \eqref{eq:sobolev},
	\begin{align*}
        \iint |\phi_k(z) - \phi_k(w)|^2 \left| K^\#_n(z, w)\right|^2 & \lesssim \iint |\phi_k(z) - \phi_k(w)|^2 n^2 e^{-n|z-w|^2 \frac{1}{4}\Delta Q(w)  } \\
        &= n^{ 1 - s_0 }\hspace{-5pt} \iint \hspace{-3.5pt}  \frac{ \left| \phi_k(z) - \phi_k(w) \right|^2 }{|z-w|^{2+2s_0}} |z-w|^{2+2s_0} n^{1+s_0} e^{-n|z-w|^2 \frac{1}{4}\Delta Q(w) } \\
		& \lesssim  n^{ 1 - s_0 } \hspace{-5pt} \iint \frac{ \left| \phi_k(z) - \phi_k(w) \right|^2 }{|z-w|^{2+2s_0}} \leq n^{1- s_0} \norm{\phi_k}_{H^{s_0}(\mathbb{R}^2)} \lesssim n^{1-s_0} \lambda_k^{s_0}.
	\end{align*}
	Notice that to eliminate any dependence on $ n $ from using the domain $ S_n $, we should rescale to $ S $ and then bound the Sobolev seminorms there, just as in the beginning of the proof.

	For the second region, start by using symmetry and equation \eqref{eq:distantzw} to get
	\begin{align*}
		\iint_{R_2} |f(z)- f(w)|^2 |K(z, w)|^2 &\lesssim \iint_{d(z, S^c)> 2 a_0 \delta_n, |z-w|> a_1 \delta_n} |f(w)|^2 |K(z, w)|^2 \\
		&\lesssim \iint_{d(z, S^c)> 2 a_0 \delta_n, |z-w|> a_1 \delta_n} |f(w)|^2 n^2 e^{- \varepsilon \sqrt{n} \min \left\{ \frac{1}{8}d(z, S^c), |z-w|  \right\} }.
    	\end{align*}
	Let's fix a constant $ b_0 $ such that $ \varepsilon b_0 = 4 $. If we were to assume that 
	\begin{equation*}
	    \min \left\{ \frac{1}{8} d(z, S^c), |z-w|  \right\} \geq b_0 \frac{\log n}{\sqrt{n}},  
	\end{equation*}
	then we could bound pointwise
	\begin{equation}\label{eq:2}
		|K(z, w)|^2 \lesssim n^2 e^{-\varepsilon \sqrt{n} \min} \leq n^2 e^{- \varepsilon b_0 \log n} = n^{-2}.
	\end{equation}
	Let's call $ \delta(z) = \frac{1}{8} d(z, S^c)  $, so we have to integrate over the conditions
	\begin{align*}
		\delta(z) > a_0 \frac{\sqrt{\log n}}{\sqrt{n}}, \quad |z-w| > a_1   \frac{\sqrt{\log n}}{\sqrt{n}},
	\end{align*}
    for some constants $a_0, a_1 >0$. Let's separate cases:
		\begin{itemize}
	    \item If $ \delta(z) < |z-w| $, then $ \min = \delta(z) $. 
	\begin{itemize}
		\item If $ \delta(z) > b_0 \frac{\log n}{\sqrt{n}}  $ then we can apply \eqref{eq:2}.
		\item If $ \delta(z) \leq b_0 \frac{\log n}{\sqrt{n}}  $, then we can integrate
		\begin{align*}
			\int_{a \frac{\sqrt{\log n}}{\sqrt{n}} < \delta(z) < b \frac{\log n}{\sqrt{n}} } n^2 e^{-\varepsilon n \delta(z)} & \approx \int_{a \frac{\sqrt{\log n}}{\sqrt{n}} < 1-r < b \frac{\log n}{\sqrt{n}} } n^2 r e^{-\varepsilon \sqrt{n} (1-r)}    dr \lesssim n^{ \frac{3}{2} - \varepsilon b_0 } = n^{ - 5/2}.
		\end{align*}	
	\end{itemize}
	\item If $ \delta (z) > |z-w| $, then $ \min = |z-w| $.
	\begin{itemize}
	    \item If $ |z-w| > b_0 \frac{\log n}{\sqrt{n}}  $ then we can apply \eqref{eq:2}.
	    \item If $ |z-w| \leq b_0 \frac{\log n}{\sqrt{n}} $ then after switching the order of integration we can bound 
	\begin{align*}
		\int_{a_1 \frac{\sqrt{\log n}}{\sqrt{n}} < |z-w| < b_0 \frac{\log n}{\sqrt{n}}  } n^2 e^{- \varepsilon \sqrt{n} |z-w|} dz &= 2 \pi \int_{a_1 \frac{\sqrt{\log n}}{\sqrt{n}} }^{b_0 \frac{\log n}{\sqrt{n}} } n^2 r e^{- \varepsilon \sqrt{n} r} dr  \lesssim n^{- 5/2}.
	\end{align*}

	\end{itemize}
	
	\end{itemize}
	All in all, we have the estimate
	\begin{equation*}
		\mathbb{E} \left|  \mu(\phi_k) - \frac{\mu(S_n)}{\sigma(S_n)} \sigma(\phi_k)   \right|^2 \lesssim \frac{\lambda^{ s_0 }}{n^{1+s_0}}.  
	\end{equation*}
    Applying Corollary \ref{cor:smoothing} then completes the proof. 
\end{proof}

\begin{proof}[Proof of Theorem \ref{thm:RNMbulk}]
    The argument is extremely similar to that of the previous Lemma, as for $n$ big enough we have that $\Omega \subset S_n$ is in the bulk. First, we need to distinguish two cases, depending on the number of random points in $ \Omega $.

    First, if $ \mu(\Omega) \leq \frac{1}{2} \mathbb{E} \mu(\Omega)  $, it is easy to check that
    \begin{equation*}
        \mathbb{P} \left( \mu(\Omega) \leq \frac{1}{2} \mathbb{E} \mu (\Omega)  \right) \lesssim \frac{1}{n}.  
    \end{equation*}
    Hence this case can be bounded in expectation. For the second case, $ \mu(\Omega) \geq \frac{1}{2} \mathbb{E} \mu(\Omega)  $, we want to apply the heat smoothing. This means we have to study for $ \phi_k $ any eigenfunction on supported on $ \Omega $ 
    \begin{align*}
        \mathbb{E} \left| \frac{\mu(\phi_k)}{\mu(\Omega)} - \frac{\sigma(\phi_k)}{\sigma(\Omega)}    \right|^2 \lesssim \mathbb{E} \left| \mu(\phi_k) - \frac{\mu(\Omega)}{\sigma(\Omega)} \sigma(\phi_k)  \right|^2,
    \end{align*}
    which can be bounded with the same fractional Sobolev technique.
\end{proof}

\begin{lemma}\label{lemma:RNM4}
    We have $\mathbb{E} W_2 (\mu_3, \sigma) \lesssim \frac{1}{\sqrt{n}}$. 
\end{lemma}

\begin{proof}[Proof of Lemma \ref{lemma:RNM4}]
    We are computing $\mathbb{E} W_2 \left( \mu(\mathbb{C} \backslash S_n) \widehat{\lambda} + \frac{\mu(S_n)}{\sigma(S_n)} \sigma|_{S_n}, \sigma \right)$. Because $\sigma(z) \gtrsim dz$ is bounded below when $z\in S$,
    \begin{equation*}
        W_2(\mu_3, \sigma) \lesssim \sup \left\{ \left| \int f d(\mu_3 - \sigma) \right| \;:\; \norm{f}_{H^{1}(S)} \leq 1  \right\}. 
    \end{equation*}
    We cannot directly use \eqref{eq:dynamicWass} as that bound would not be precise enough on the edge. Let's write $ \mu_3- \sigma = \gamma_{1} + \gamma_{2} + \gamma_{3} $:
    \begin{align*}
	    \gamma_1 & = \left( \mu(\mathbb{C} \backslash S_n)- \sigma(E) \right) \widehat{\lambda},\\
	    \gamma_2 & = \sigma(E) \widehat{\lambda} - \sigma|_E, \\
	    \gamma_3 & = \left( \frac{\mu(S_n)}{\sigma(S_n)} - 1  \right) \sigma|_{S_n}. 
    \end{align*}
    The measure $ \gamma_3 $ corresponds with the part inside the bulk, and can be well controlled by an $ L^2 $-type bound. The other two signed measures are all supported on the edge $ E $, where we will need to exploit that although it has a fixed diameter, it is a thin layer. Notice that $ \gamma_2 $ is not a random measure.

    Let's start with $ \gamma_1 $. Denote the set around the boundary $ \Omega_A = \left\{ z\in S \;:\; d(z, \partial S) < A \right\}  $. Using the coarea formula we have
    \begin{align*}
	    \int_{\Omega_A} |f|^2 = \int_{0}^A \int_{d(z, \partial S)} |f(z)|^2 dz dt = \int_0^A \norm{f}^2_{L^2(\Gamma_t)} dt.
    \end{align*}
    We are integrating over the curves $ \Gamma_t = \left\{ z \in S\;:\; d(z, \partial S) = t\right\}  $. We can use the trace theorem to get
    \begin{align*}
	    \norm{f}_{L^2(\Gamma_t)} \leq C_t \norm{f}_{L^2(\Omega_A)} \leq C_t \norm{f}_{L^2(S)} \lesssim C_t \norm{f}_{H^1(S)},
    \end{align*}
    where we have used the Poincare inequality on $ S $, as we can assume without loss of generality that $ \int f = 0 $. Finally, it is well-known that for $ t $ small enough, the constants $ C_t $ are uniformly bounded, as the curves $ \Gamma_t $ are smooth small perturbations of $ \partial S $. Combining this two results,
    \begin{align*}
	    \int_{\Omega_A} |f|^2 \lesssim A. 
    \end{align*}
    Choosing $ A \propto \delta_n $ so that $ E \subset \Omega_A $, we have 
    \begin{align*}
	    \left| f d\gamma_1 \right| \leq \norm{f}_{L^2(\Omega_A)} \norm{ \gamma_1 }_{L^2(E)} \lesssim \sqrt{\delta_n} \sqrt{|E|} \norm{\gamma_1}_{L^\infty} \approx \delta_n \norm{\gamma_1}_{L^\infty} \lesssim \delta_n \left| \mu(\mathbb{C} \backslash S_n) - \sigma(E) \right| \delta_n^{-1} .
    \end{align*}
    Now, taking expectation,
    \begin{align*}
	    \mathbb{E} \left| \mu(\mathbb{C} \backslash S_n) - \sigma(E) \right| = \mathbb{E} \left| \mu(S_n) - \sigma(S_n) \right| \lesssim \sqrt{\Var \mu(S_n)} + \left| \mathbb{E} \mu(S_n) - \sigma(S_n) \right| \lesssim \frac{1}{\sqrt{n}}.  
    \end{align*}
    For $ \gamma_2 $, notice that we can write the measures in a tubular coordinates of $ \partial S $,
    \begin{align*}
	    \sigma(E)\widehat{\lambda} = \frac{\sigma(E)}{\rho(E)} m(s) \operatorname{Unif}(F_s), && \sigma|_E = \sigma_{F}(s) p_s,
    \end{align*}
    where $ p_s $ is a probability measure on the fiber $ F_s $ with density proportional to $ \sigma $. $ m(s) $ and $ \sigma_{F}(s) $ represent the ``mass'' that each measure has on the fiber. We further split 
    \begin{align*}
	    \gamma_2 = \gamma_{s} + \gamma_m = \frac{\sigma(E)}{\rho(E)} m(s) \left[ \operatorname{Unif}(F_s) - p_s \right] ds + \left[ \frac{\sigma(E)}{\rho(E)} m(s) - \sigma_{F}(s) \right] p_s ds 
    \end{align*}
    $ \gamma_s $ encodes the information regarding the difference in shape of the densities in each fiber, while $ \gamma_m $ will measure the discrepancy between the masses of each fiber. 
    For each $ s $, $ \left[ \operatorname{Unif}(F_s) - p_s \right] $ is a signed measure on $ F_s = \left\{ s \right\} \times [0, \tau(s)] $. It has mass $ 0 $, and both measures $ p_s  $ and $ \operatorname{Unif}(F_s) $ have densities bounded by $ \lesssim \frac{1}{\tau(s)} \approx \delta_n^{-1}$. Consider the function 
    \begin{align*}
	    V(s, t) := \int_{t}^{\tau(s)} \frac{\sigma(E)}{\rho(E)} m(s) \left[ \frac{1}{\tau(s)} - p_s(u)  \right] du, \quad V(s, 0) = V(s, \tau(s)) = 0, \quad |V(s, t)| \lesssim \delta_n.
    \end{align*}
    Use then integration by parts 
    \begin{align*}
	    \int_{F_s} f(s, t) \frac{\sigma(E)}{\rho(E)} m(s) \left[ \frac{1}{\tau(s)} - p_s(t)  \right] dt = -\int_{F_s} \partial_t f(s, t) V(s, t) dt.
    \end{align*}
    This gives the bound
    \begin{align*}
	    \left| \int f d \gamma_s \right| &= \left| \int_E \partial_t f V(s, t) \left( 1 + \kappa(s)t \right) dsdt   \right|  \lesssim \norm{\partial_t f}_{L^2(E)} \norm{V}_{L^2(E)} \lesssim \norm{\nabla f}_{L^2} \delta_n^{ \frac{3}{2} } \lesssim \frac{1}{\sqrt{n}}.
    \end{align*}
    For the $ \gamma_m $ term, notice that using \eqref{eq:edge}, we can bound 
    \begin{align*}
	    m(s) - \sigma_{F}(s) &= \int_{t^-(s)}^{t^+(s)} \left( \frac{1}{n}K(z, z) - \sigma(z)   \right) \left( 1 + \kappa(s) t \right) dt  \\
	    \left| m(s) - \sigma_F(s) \right| & \lesssim \int_{t^-(s)}^{t^+(s)} \left( e^{-cn d(z)^2} + \frac{\log^3 n}{\sqrt{n}}   \right) dt \lesssim \frac{1}{\sqrt{n}}. 
    \end{align*}
    This means that 
    \begin{align*}
	    \left| \left(\frac{\sigma(E)}{\rho(E_n)} -1 \right) m(s) + m(s) - \sigma_F(s)  \right| \leq \frac{\left| \rho(E_n) - \sigma(E) \right| }{\rho(E_n)} m(s) + \left| m(s) - \sigma_F(s) \right| \lesssim \frac{1}{\sqrt{n}},    
    \end{align*}
    where it is used Lemma \ref{lemma:densitym} to cancel out $ m(s) \approx \delta_n \approx \rho(E_n) $ and that $\left| \rho(E_n) - \sigma(E) \right| \lesssim \frac{1}{n}$. We can then bound
    \begin{align*}
	    \left| \int f d \gamma_m \right| = \left| \int \left[ \frac{\sigma(E)}{\rho(E_n)} m(s) - \sigma_F(s)  \right] f(s, t) p_s(t) dtds  \right| \lesssim \frac{1}{\sqrt{n}}   \left| \int f(s, t) p_s(t) dtds  \right|.
    \end{align*}
    In each fiber, using that $ p_s $ has mass 1,
    \begin{align*}
	    F(s) := \left| \int f(s, t) p_s dt \right| \leq \left| \int \left( f(s, t) - f(s, 0) \right) p_s dt  \right| + |f(s,0)|  \leq \norm{\partial_t f}_{L^1(F_s)} + |f(s, 0)|.
    \end{align*}
    Noticing that $ f(s, 0) $ is just the function evaluated on $ z\in \partial S $, we get the estimate
    \begin{align*}
	    \int |F(s)| ds \leq \norm{\partial_t f }_{L^1(E)} + \norm{f}_{L^1(\partial S)} \lesssim \norm{\partial_t f }_{L^2(E)} + \norm{f}_{L^2(\partial S)} \lesssim 1.
    \end{align*}
    
    For $ \gamma_3 $, we directly use Poincare to get 
    \begin{align*}
	    \left| f d \gamma_4 \right| \leq \norm{f}_{L^2} \norm{\gamma_4}_{L^2} \lesssim \norm{\gamma_4}_{L^2} \lesssim \left| \frac{\mu(S_n)}{\sigma(S_n)} - 1  \right|.
    \end{align*}
    When taking expectation we would have to study $ \mathbb{E} \left| \mu(S_n) - \sigma(S_n) \right|  $, which already has been bounded in the $ \gamma_1 $ argument. \vspace{-0.1em}
\end{proof}

\section*{Acknowledgments}
The author has been supported by the grant ``Ayudas para contratos predoctorales para la formación de doctores 2022'', by the Agencia Estatal de Investigación and by the FSE+, project PID2021-123405NB-I00; and by the project PID2024-160033NB-I00.

\textbf{LLM usage disclosure:} Large Language Models were used as a supplementary tool during the development of this manuscript, primarily in the process of revision and troubleshooting. All AI-generated suggestions were treated strictly as advisory, and the author rigorously reviewed, adapted, and mathematically verified all outputs independently, taking full responsibility for the final content and validity of this work.

\printbibliography

@article{peyre_2018, title={Comparison between {W2} distance and {H-1} norm, and {L}ocalization of {W}asserstein distance}, volume={24}, DOI={https://doi.org/10.1051/cocv/2017050}, number={4}, journal={ESAIM: Control, Optimisation and Calculus of Variations}, author={Peyre, Rémi}, year={2018},  pages={1489-1501} }

@phdthesis{prodhomme,
  TITLE = {{Contributions to the optimal transport problem and its regularity}},
  AUTHOR = {Prod'Homme, Maxime},
  URL = {https://theses.hal.science/tel-03419872},
  NUMBER = {2021TOU30122},
  SCHOOL = {{Universit{\'e} Paul Sabatier - Toulouse III}},
  YEAR = {2021},
  TYPE = {Theses},
  HAL_ID = {tel-03419872},
  HAL_VERSION = {v2},
}

@article{arias2024equidistributionpointsharmonicensemble,
author = {García Arias, Pablo},
title = {Equidistribution of points in the harmonic ensemble for the Wasserstein distance},
journal = {Mathematika},
volume = {72},
number = {2},
pages = {e70095},
doi = {https://doi.org/10.1112/mtk.70095},
year = {2026}
}

@article{borda2023riesz,
author = {Borda, Bence and Grabner, Peter and Matzke, Ryan W.},
title = {Riesz energy, discrepancy, and optimal transport of determinantal point processes on the sphere and the flat torus},
journal = {Mathematika},
volume = {70},
number = {2},
pages = {e12245},
doi = {https://doi.org/10.1112/mtk.12245},
year = {2024}
}

@article{WardIdentities, title={Random normal matrices and {W}ard identities}, volume={43}, DOI={https://doi.org/10.1214/13-aop885}, number={3}, journal={The Annals of Probability}, publisher={Institute of Mathematical Statistics}, author={Yacin Ameur and Haakan Hedenmalm and Makarov, Nikolai}, year={2015} }

@article{ameur_2011, title={Near-{B}oundary {A}symptotics for {C}orrelation Kernels}, volume={23}, DOI={https://doi.org/10.1007/s12220-011-9238-4}, number={1}, journal={Journal of Geometric Analysis}, author={Ameur, Yacin}, year={2011},  pages={73–95} }

@book{santambrogio_2015, title={Optimal {T}ransport for {A}pplied {M}athematicians}, ISBN={9783319208282}, publisher={Birkhäuser}, author={Filippo Santambrogio}, year={2015} }

@book{villani_2009, 
    place={Berlin, Heidelberg}, 
    title={Optimal transport: {O}ld and new}, 
    publisher={Springer Berlin Heidelberg}, 
    author={Villani, Cédric}, 
    year={2009}
}

@book{Hough_Krishnapur_Peres_2012, place={Providence (R.I.)}, title={Zeros of {G}aussian analytic functions and {D}eterminantal {P}oint {P}rocesses}, publisher={American Mathematical Society}, author={Hough, John Ben and Krishnapur, Manjunath and Peres, Yuval}, year={2012}}

@article{bobkov_ledoux_2021, title={A simple {F}ourier analytic proof of the {AKT} optimal matching theorem}, volume={31}, DOI={https://doi.org/10.1214/20-aap1656}, number={6}, journal={The Annals of Applied Probability}, author={Bobkov, Sergey G. and Ledoux, Michel}, year={2021} }

@article{brown_steinerberger_2020, title={On the {W}asserstein distance between classical sequences and the Lebesgue measure}, volume={373}, DOI={https://doi.org/10.1090/tran/8212}, number={12}, journal={Transactions of the American Mathematical Society}, author={Brown, Louis and Steinerberger, Stefan}, year={2020},  pages={8943–8962} }

@article{borda_2023, title={Empirical measures and random walks on compact spaces in the quadratic {W}asserstein metric}, volume={59}, DOI={https://doi.org/10.1214/22-aihp1322}, number={4}, journal={Annales de l’Institut Henri Poincaré, Probabilités et Statistiques}, author={Borda, Bence}, year={2023} }

@misc{borda2025smoothinginequalitiestransportmetrics,
      title={Smoothing inequalities for transport metrics in compact spaces}, 
      author={Bence Borda and Jean-Claude Cuenin},
      year={2025},
      eprint={2510.21380},
      archivePrefix={arXiv},
      primaryClass={math.CA}
}

@article{Ginibre1965,
    author = {Ginibre, Jean},
    title = {Statistical {E}nsembles of {C}omplex, {Q}uaternion, and {R}eal Matrices},
    journal = {Journal of Mathematical Physics},
    volume = {6},
    number = {3},
    pages = {440-449},
    year = {1965},
    issn = {0022-2488},
    doi = {10.1063/1.1704292}
}

@article{elbau_felder_2005, title={Density of {E}igenvalues of {R}andom {N}ormal {M}atrices}, volume={259}, DOI={https://doi.org/10.1007/s00220-005-1372-z}, number={2}, journal={Communications in Mathematical Physics}, author={Elbau, Peter and Felder, Giovanni}, year={2005},  pages={433–450} }

@article{hedenmalm_makarov_2012, title={Coulomb gas ensembles and {L}aplacian growth}, volume={106}, DOI={https://doi.org/10.1112/plms/pds032}, number={4}, journal={Proceedings of the London Mathematical Society}, publisher={Wiley}, author={Håkan Hedenmalm and Makarov, Nikolai}, year={2012},  pages={859–907} }

@article{jalowyCircularLaw,
author = {Jonas Jalowy},
title = {{The {W}asserstein distance to the circular law}},
volume = {59},
journal = {Annales de l'Institut Henri Poincaré, Probabilités et Statistiques},
number = {4},
publisher = {Institut Henri Poincaré},
pages = {2285 -- 2307},
year = {2023},
doi = {10.1214/22-AIHP1317}
}

@article{MarzoUniversality,
author = {Marzo, Jordi and Molag, Leslie and Ortega-Cerdà, Joaquim},
title = {Universality for fluctuations of counting statistics of random normal matrices},
journal = {Journal of the London Mathematical Society},
volume = {113},
number = {2},
pages = {e70462},
doi = {https://doi.org/10.1112/jlms.70462},
year = {2026}
}

@article{BerezinTransform,
author = {Ameur, Yacin and Makarov, Nikolai and Hedenmalm, Håkan},
title = {Berezin transform in polynomial bergman spaces},
journal = {Communications on Pure and Applied Mathematics},
volume = {63},
number = {12},
pages = {1533-1584},
doi = {https://doi.org/10.1002/cpa.20329},
year = {2010}
}

@book{mclean_2000, address={Cambridge}, title={Strongly {E}lliptic {S}ystems and {B}oundary {I}ntegral {E}quations}, ISBN={9780521663755}, publisher={Cambridge University Press}, author={McLean, William}, year={2000} }

@article{BREZIS20181355,
title = {{G}agliardo–{N}irenberg inequalities and non-inequalities: The full story},
journal = {Annales de l'Institut Henri Poincaré C, Analyse non linéaire},
volume = {35},
number = {5},
pages = {1355-1376},
year = {2018},
issn = {0294-1449},
doi = {https://doi.org/10.1016/j.anihpc.2017.11.007},
author = {Haïm Brezis and Petru Mironescu}
}

@book{coste_2021, title={Order, fluctuations, rigidities: work in progress}, url={https://scoste.fr/assets/survey_hyperuniformity.pdf}, author={Coste, Simon}, year={2021} }

@article{CHAFAI20181447,
title = {Concentration for {C}oulomb gases and {C}oulomb transport inequalities},
journal = {Journal of Functional Analysis},
volume = {275},
number = {6},
pages = {1447-1483},
year = {2018},
issn = {0022-1236},
doi = {https://doi.org/10.1016/j.jfa.2018.06.004},
author = {Djalil Chafaï and Adrien Hardy and Mylène Maïda}
}

@misc{hyperuniformityoptimaltransportpoint,
      title={Hyperuniformity and optimal transport of point processes}, 
      author={Raphaël Lachièze-Rey and D. Yogeshwaran},
      year={2026},
      eprint={2402.13705},
      archivePrefix={arXiv},
      primaryClass={math.PR},
}

@article{ERBAR2025110974,
title = {Optimal transport of stationary point processes: {M}etric structure, gradient flow and convexity of the specific entropy},
journal = {Journal of Functional Analysis},
volume = {289},
number = {4},
pages = {110974},
year = {2025},
issn = {0022-1236},
doi = {https://doi.org/10.1016/j.jfa.2025.110974},
author = {Matthias Erbar and Martin Huesmann and Jonas Jalowy and Bastian Müller}
}

@misc{nazarov2010fluctuationsrandomcomplexzeroes,
      title={Fluctuations in random complex zeroes: {A}symptotic normality revisited}, 
      author={Fedor Nazarov and Mikhail Sodin},
      year={2010},
      eprint={1003.4251},
      archivePrefix={arXiv},
      primaryClass={math.PR}
}

\end{document}